\documentclass[final]{siamltex}

\usepackage{todonotes}%\usepackage{refcheck}

\usepackage{amsmath,amstext,amsbsy,amssymb,graphicx,subfig,mathdots}
\usepackage{algorithm,algpseudocode}
\newenvironment{remark}{{\bf Remark}}

\def\D{{\cal D}}
\def\B{{\cal B}}
\def\L{{\cal L}}
\def\O{{\cal O}}
\def\M{{\cal M}}
\def\S{{\cal S}}
\def\K{{\cal K}}
\def\R{{\cal R}}
\def\P{{\cal P}}
\def\Ai{{\text{Ai}}}

\def\qqand{\qquad\hbox{and}\qquad}
\def\dkf{{\mathrm d}}
\def\dx{\dkf x}
\def\half{{\frac{1}{2}}}

\title{A fast and well-conditioned spectral method}

\author{Sheehan Olver\thanks{School of Mathematics and Statistics, The
University of Sydney, Sydney, Australia. (Sheehan.Olver@sydney.edu.au)}
        \and Alex Townsend\thanks{Mathematical Institute,
24-29 St Giles', Oxford, England, OX1 3LB. (townsend@maths.ox.ac.uk)}}

\begin{document}

\maketitle

\begin{abstract}
A spectral method is developed for the direct solution of linear
ordinary differential equations with variable coefficients. The method
leads to matrices which are almost banded, and a numerical solver is
presented that takes
$\O(m^2n)$ operations, where $m$ is the number of Chebyshev points
needed to resolve the coefficients of the differential operator and 
$n$ is the number of Chebyshev coefficients needed to resolve the
solution to the differential equation.  We prove stability of the method
by relating it to a diagonally preconditioned system which has a bounded
condition number, in a suitable norm.  For Dirichlet boundary conditions,
this implies stability in the standard $2$-norm.  An adaptive QR factorization is developed to efficiently solve the resulting linear system and automatically choose the optimal number of Chebyshev coefficients needed to represent the solution.  The resulting algorithm can efficiently and reliably solve for solutions that require as many as a million unknowns.  
\end{abstract}



\pagestyle{myheadings}
\thispagestyle{plain}
\markboth{}{}

\section{Introduction}
Spectral methods are an important tool in scientific computing and engineering
 for solving differential equations  (see, for instance,
\cite{Canuto_06_01,Fornberg_98_01,Gottlieb_98_01,Shen_09_01,Trefethen_00_01}). Although the
computed solutions can converge super-algebraically to the solution of the
differential equation, conventional
wisdom states that spectral methods lead to dense, ill-conditioned matrices. In
this paper, we introduce a spectral method which continues to
converge super-algebraically to the solution, but only requires solving an
almost banded, well-conditioned linear system. 

Throughout,  we consider the family of linear
differential equations on $[-1,1]$:
	\begin{equation}\label{ODE}
		\L u =  f \qqand \B u = {\mathbf c}
	\end{equation}
where $\L$ is an $N$th order linear differential operator 
	$$\L u = a^N(x) \frac{\dkf^N u}{ \dx^N} + \cdots + a^1(x) \frac{\dkf u}
{\dx} + a^0(x) u,$$
 $\B$ denotes $K$ boundary conditions (Dirichlet, Neumann, etc.), ${\mathbf c}
\in {\mathbb C}^K$ and
$a^0,\ldots,a^N$, $f$ are suitably smooth functions on $[-1,1]$.  We make
the
assumption that the differential equation is not singular; i.e., the leading variable coefficient $a^N(x)$ does not vanish on the interval $[-1,1]$.

	Within spectral methods there is a
subdivision between
collocation methods and coefficient methods; the former construct matrices 
operating on the values of a function at, say, Chebyshev points; the latter
construct matrices operating on coefficients in a basis, say, of Chebyshev
polynomials.
Here there is a common belief that collocation methods are more adaptable to
differential equations with variable coefficients \cite[section 9.2]{Boyd_01_01}; i.e., when
$a^0(x),\ldots,a^N(x)$ are not constant. 
 However, the spectral coefficient method that we construct
is equally applicable to differential equations with variable coefficients.

%(Typically $K = N$, though $K$ 
%may  be less than $N$ if, say,  $a^N(x)$ vanishes in $[-1,1]$.  Issues with
% nontrivial kernels allow for $K$ to be greater than $N$.)
	
% Three important approaches to numerically compute the solution to
% \eqref{ODE} are finite element, finite difference and spectral methods.  The
% conventional wisdom
% states that they have the following trade-offs: finite element and
% finite difference methods lead to
% sparse matrices but, typically, converge algebraically, while spectral
% methods have
% dense matrices
% but converge super-algebraically.  Moreover, in each case the condition number
% of
% the
% matrices increases, resulting in limited accuracy of the approximation. 
% Finally, 
% the accuracy of each approach breaks down when applied to higher order
% differential equations.  \sotodo{Is there a good
% reference?  Maybe this is well-known enough that we can get away with it as
% is. 
% Look at Fornberg} 
		
% 	The numerical method begins with a simple observation, which has been
% utilized for the construction of other coefficient methods when the
% differential
% equation only has constant coefficients \cite{Shen_04_01}.  

For example, the first order differential equation (chosen so the
entries in the resulting linear system are integers)
\begin{equation}\label{firstorder}
	{\frac{\dkf u}{\dx}}	+ 4xu = 0 \qqand u(-1) = c,
\end{equation}
results in our spectral method forming the {\em almost banded} $n\times n$ linear
system

\begin{equation}
\left(
\begin{array}{c c c c c c c c}
1 & -1 & 1 & -1 & 1 & -1 &\cdots  & (-1)^{n-1}\\[4pt]
  & 2  &   & -1 &   &    &        &           \\[4pt]
2 &    & 2 &    &  -1 &    &        &           \\[4pt]
  & 1  &   &  3 &   & -1 &        &           \\[4pt]		
  &    &\smash{\ddots}&  &\smash{\ddots}&& \smash{\ddots}&\\[4pt]
  &    &   & 1  &   &n-3 &        &         -1\\[4pt]
  &    &   &    & 1 &    & n - 2  &           \\[4pt]
  &    &   &    &   &  1 &        &  n - 1    \\
\end{array} \right)
\left(\begin{array}{c}u_0\\[4pt]
u_1\\[4pt]u_2\\[4pt]\\[4pt]\smash{\vdots}\\[4pt]
\smash{\vdots}\\[4pt]\\[4pt]u_{n-1}\end{array}\right)
=
\left(\begin{array}{c}c\\[4pt]0\\[4pt]0\\[4pt]\\[4pt]\smash{\vdots}\\[4pt]\smash
{\vdots}\\[ 4pt ] \\[4pt]0\\ \end{array}\right).
\label{firstsystem}
\end{equation}
We then approximate the solution to (\ref{firstorder}) as 
\[
u(x) = \sum_{k=0}^{n-1} u_k T_k(x)
\]
where $T_k$ is the Chebyshev polynomial of degree $k$ (of the first kind).
Moreover, the stability of solving (\ref{firstsystem}) is directly related to that of a diagonally preconditioned matrix system which has
$2$-norm condition number
bounded above by $53.6$ for all $n$.

%The preconditioned matrix, just like
%(\ref{firstsystem}), is almost banded and it can be solved in
%$\O(n)$ operations (as discussed in Section \ref{sec:fastsolver}).

	Our 
	method is based on: 
\begin{enumerate}
\item Representing derivatives in terms of {\it  ultraspherical polynomials}. This results in diagonal differentiation matrices.
\item Representing conversion from Chebyshev polynomials to ultraspherical polynomials by a {\it banded} operator.
\item Representing multiplication for variable coefficients by banded operators in coefficient space.  This is achieved by approximating the variable coefficients by a truncated Chebyshev (and thence ultraspherical) series.  
\item Imposing the boundary conditions using {\it boundary bordering}, that is, $K$ rows
of the linear system are used to impose $K$ boundary conditions.  For
(\ref{firstorder}), the
last row of the linear system has been replaced and then permuted to the first
row.  This allows for a convenient solution basis to be used for general boundary conditions.
\item Using an {\it adaptive QR decomposition} to determine the optimal truncation denoted by $n_{\rm opt}$.  We develop a sparse representation for the resulting dense matrix, allowing for efficient back substitution.
\end{enumerate} 
The resulting stable method requires only $\O(n_{\rm opt})$ operations and calculates the Chebyshev coefficients of the solution (as well as its derivatives) to machine precision.  A striking application of the method is to singularly perturbed differential equations, where the bandwidth of the linear system is uniformly bounded, and the accuracy is unaffected by the extremely large condition numbers.

This paper is organized as follows.   We first provide an overview of existing  spectral methods.  In section~\ref{sec:firstorder}, we construct the method for first order differential equations and apply it to two
problems with highly oscillatory variable coefficients. In the third section,
we extend the approach to higher order differential equations
by using ultraspherical polynomials, and present numerical results of the method applied to 
challenging second and higher-order differential equations. In the fourth
section, we prove
stability and convergence of the method in high order norms, which reduce to the 
standard $2$-norm for Dirichlet boundary
conditions. In section~\ref{sec:fastsolver}, we present a fast, stable algorithm to solve the almost banded linear systems in
$\O(n_{\rm opt})$ operations, where $n_{\rm opt}$ is automatically determined, allowing for the efficient solution of differential equations that require as many as a million unknowns (see Figure~\ref{fig:adaptiveqrtiming}).  In the final section we describe directions for future research.

\begin{remark}
	The {\sc Matlab} and {\sc C++} code used for the numerical results is available from \cite{Code}.  
\end{remark}

\subsection{Existing techniques}
Chebyshev polynomials or, more generally, Jacobi polynomials have been
abundantly used to construct spectral coefficient methods. In this section we briefly survey some existing techniques.

\subsubsection{Tau-method}
The tau-method was originally proposed by Lanczos \cite{Lanczos_38_01} and is analogous to representing differentiation as an operator on Chebyshev coefficients \cite{Boyd_01_01}.  The resulting linear system, which is constructed by using recurrence relationships between Chebyshev polynomials, is always dense --- even for differential equations with constant variable coefficients --- and is typically ill-conditioned.  The boundary conditions are imposed by replacing the last few rows of the linear system with entries constraining the solution's coefficients. 

This method was made popular and extended by Ortiz and his colleagues \cite{Gottlieb_98_01,Ortiz_69_01} and can be made into an automated black-box solver, but this is more due to how the boundary conditions are imposed rather than the specifics of this approach.  Therefore, we will use the same technique called {\em boundary bordering} for the boundary conditions.

%The spectral collocation method, as implemented in
%Chebfun \cite{Chebfun}, automatically resolves a solution to a differential equation,
%but suffers from ill-conditioning and also involves inverting dense matrices \cite{Driscoll_08_01}. 

\subsubsection{Basis recombination}

An alternative to boundary bordering is to impose the boundary conditions by \emph{basis
recombination}.  That is, for our simple example (\ref{firstorder}), by computing the
coefficients in the expansion  
\[
u(x) = cT_0(x) + \sum_{k=1}^\infty \tilde{u}_k \phi_k(x)
\]
where
\[
\phi_k(x) = 
\begin{cases}
T_k(x) - T_0(x) & k \text{ even}\\
T_k(x) + T_0(x) & k \text{ odd}.\\
\end{cases}
\]
The basis is chosen so that $\phi_k(-1)=0$, and there are many other possible
choices.

Basis recombination runs counter to an important observation of
Orszag \cite{Orszag_71_01}: Spectral methods in a convenient basis (like
Chebyshev) can outperform an awkward problem-dependent basis. The
problem-dependent basis may be theoretically elegant, but convenience is worth
more in practice.  
In detail, the benefit of using boundary bordering instead of basis recombination include: (1) the solution is always computed in the convenient
and orthogonal Chebyshev basis; (2) a fixed basis means we can automatically
apply recurrence
relations between Chebyshev polynomials (and later, {ultraspherical
polynomials}) to construct multiplication matrices for incorporating variable coefficients; (3) the structure of the linear systems does
not depend on the boundary condition(s), allowing for a fast, general solver; and (4) while basis recombination can result in well-conditioned linear systems, the solution can be in terms of an unstable basis for very high order boundary conditions.  

There is one negative consequence of boundary bordering: it results in non-symmetric matrices --- even for self-adjoint
differential operators --- potentially doubling the computational cost of solving the resulting linear system when using direct methods.  
However,  our focus is on general differential equations, which are not necessarily self-adjoint.  

\subsubsection{Petrov--Galerkin methods}
Petrov--Galerkin methods solves linear ODEs using different bases for representing the solution (trial basis) and the right-hand side  of the equation (test basis). %\attodo{Not quite true}
% In terms of operators it means that differentiation is an operator from one basis to a different one. 
 The boundary conditions determine the trial basis \cite{Ma_00_01}, and the boundary conditions associated to the dual differential equation can determine the test basis \cite{Shen_95_01}.   Shen has constructed bases for second, third, fourth and higher
odd order differential equations with constant coefficients so that the
resulting matrices are banded or
easily invertible \cite{Doha_09_01,Shen_95_01,Shen_04_01,Shen_07_01}. Moreover, it was shown that very
specific variable coefficients preserve this matrix structure
\cite{Shen_09_01}.  These methods can be very efficient, but the method for imposing the boundary conditions is difficult to automate. 

Our method will also use two different bases: the Chebyshev and ultraspherical
polynomial bases.  This choice preserves sparsity as do those considered in
\cite{Doha_09_01,Ma_00_01,Shen_95_01,Shen_07_01}, but our bases will depend on the order of the differential equation and not the boundary conditions. %will not depend on the boundary conditions.

%\sotodoinline{More accurate than Galerkin?}

\subsubsection{Integral reformulation}

Integral reformulation expresses the solution's $N$th derivative in a Chebyshev expansion, solves for those coefficients before integrating back to obtain the coefficients of the solution \cite{Coutsias_96_01,Greengard_91_01,Zebib_83_01}.
 The same idea was previously advocated by Clenshaw \cite{ClenshawTauMethod} and a comparison was made to the tau-method by Fox \cite{Fox_62_01} who concluded that Clenshaw's method was more accurate.   In the case of constant coefficients, an alternative approach is to integrate   the differential equation itself, which results in a banded system \cite{Viswanath_12_01}.    Recently, integration reformulation has received considerable
attention in the literature because it can construct well-conditioned linear systems \cite{Driscoll_10_01,Lee_97_01,Viswanath_12_01}.

%		where the differentiation operators are replaced with banded integration operators, and variable coefficients are represented by banded multiplication operators.  However, our approach represents differentiation as a banded operator, which has one crucial advantage over integral reformulation: it can be trivially adapted to operators of the form
%%
%	$${\dkf \over \dx} a(x) {\dkf  u\over \dx}.$$
%%
%Integral reformulation will not annihilate the outer differentiation, which corresponds to a dense operator.    Replacing the operator with
%%
%	$$a(x) {\dkf^2  u\over \dx^2} + a'(x) {\dkf  u\over \dx}.$$
%%
%is not as general, and the required numerical differentiation would introduce numerical errors.   The method we describe can also be used to solve integro-differential equations such as
%\begin{equation}
% a(x)u'(x) + b(x)u(x) + \int_{-1}^x c(s)u(s)ds = f(x),
%\label{eq:integrodiff}
%\end{equation}
%resulting in an almost banded linear system because both differentiation and integration are sparse operators. 

	Clenshaw's method was a motivating example for (F. W. J.) Olver's algorithm \cite{OlversAlgorithm,Lozier_80}: a method for choosing the optimal truncation parameter $n_{\rm opt}$ for calculating solutions to recurrence relationships (i.e., banded linear systems).  The original paper considered dense boundary rows to handle the boundary conditions for Clenshaw's method.      In section~\ref{sec:opttrunc}  we develop the adaptive QR decomposition which is a generalization of Olver's algorithm.  The original approach was based on adaptive Gaussian elimination without pivoting, with a proof of numerical stability for sufficiently large $n$, but, by replacing LU with QR, we achieve stability for all $n$.

		Many of the ideas that we develop in this paper are equally applicable to integral reformulation, as the resulting matrices are also almost banded.  However, imposing boundary conditions requires additional variables related to the integration constants, considerably complicating the representation of  multiplication.  Moreover, integral reformulation does not maintain sparsity for {\it mixed operators} such as
	$${\dkf \over \dx} a(x) {\dkf  u\over \dx},$$
as the integration will no longer  annihilate the outer differentiation, which corresponds to a dense operator.  Our approach avoids such issues.

\subsubsection{Preconditioned iterative methods}

Another approach to obtain well-conditioned linear systems is preconditioning, with the most common being motivated by a finite difference
stencil \cite{Orszag_80_01}, the operator representing the ODE with all its variable coefficients set to constants \cite{Bender_78_01,Shen_09_01}, a finite element approximation
\cite{Canuto_85_01,Deville_85_01} or integration preconditioners
\cite{Hesthaven_98_01,Elbarbary_06_01}.
Once preconditioned, iterative solvers can be
employed along with the Fast Fourier Transformation, which can be more efficient than dense solvers in certain situations \cite{Boyd_01_01,Canuto_06_01,Shen_09_01}.   While specially designed iterative methods can outperform direct methods for specific problems, they do not provide the generality and robustness that we require.  

%Preconditioning spectral matrices before using an iterative is known as the \emph{preconditioned iterative method} and will be discussed further in section \ref{sec:furtherdiscuss}.

Differential equations with highly oscillatory or nonanalytic variable coefficients are one example where  iterative methods may be necessary to achieve computational efficiency, as the multiplication operators have very large bandwidth in coefficient space, resulting in essentially dense linear
systems.  
Iterative methods based on
matrix-vector products avoid this growth in complexitiy, hence can be more efficient.
% The matrix-vector product of a $n \times n$
%discretization of a differential operator can take just
%$\mathcal{O}(n\log n)$ operations and an iterative method can have near optimal complexity. 
 However, the matrix-vector product has to be user-supplied code which is unlikely to be optimized in terms of memory caching or memory allocation; therefore,  significant benefits over direct methods will only appear for problems that require large $n$  \cite{Weideman_00_01}.   Moreover, direct methods provide better stability and robustness \cite{Greengard_91_01}.  Finally, while the computational efficiency of our approach is lost for large bandwidth variable coefficients, numerical stability is maintained,  as demonstrated in the second example of section \ref{sec:firstorderexamples}.

 % The preconditioned iterative method is still extremely useful when $n$ is large and direct solvers become inefficient.  

%We use the Chebyshev polynomial basis and an ultraspherical polynomial basis,
%as we will explain later.

%We mention that the solution of linear equations with variable coefficients is
%critical to the solution of nonlinear problems.  
%While
%we focus on linear equations in this paper, numerical experiments confirm that 
%the approach is applicable to nonlinear problems as well, continuing to achieve stability.  However,
%the diagonal form of variable coefficients in collocation methods may prove computationally more effective for certain problems, though at the expense of accuracy.    In the final section we discuss future directions that may overcome this deficiency.  

%We
%construct an
%efficient method that automatically resolves the solution of a general linear
%variable coefficient differential equation. Moreover, for Dirichlet boundary
%conditions we use a diagonal preconditioner so that the resulting linear system
%has bounded condition number.  Later, we show that we achieve bounded condition
%number for Neumann and Robin boundary conditions, but not in the $\ell_2$-norm.
%We argue that a straightforward modification of the numerical algorithm is
%stable in these higher order norms, ensuring backward stability.

\section{Chebyshev polynomials and first order differential equations}\label{sec:firstorder}
For pedagogical reasons, we begin by solving first-order differential
equations of the form 
\begin{equation}
u'(x) + a(x)u(x) = f(x) \qqand u(-1)=c
\label{eq:firstDE}
\end{equation}
where $a : [-1,1]\rightarrow \mathbb{C}$ and $f:[-1,1]\rightarrow\mathbb{C}$
are continuous functions with bounded variation. The continuity assumption
ensures that (\ref{eq:firstDE}) has a unique continuously
differentiable solution on the unit interval \cite{Ritger_00_01}, while bounded variation ensures a unique
representation as a uniformly convergent Chebyshev expansion
\cite[Thm. 5.7]{Mason_03_01}.  An exact representation of a continuous function $g(x)$ with bounded variation is 
\begin{equation}\label{eq:aChebSeries}
g(x)  = \sum_{k=0}^\infty g_kT_k(x), \hbox{  }
 g_0 = \frac{2}{\pi}\int_{-1}^1\frac{g(x)}{\sqrt{1-x^2}} \dx, \hbox{  } g_k = \frac{1}{\pi}\int_{-1}^1\frac{g(x) T_k(x)}{\sqrt{1-x^2}} \dx %\notag
\end{equation} 
where  $T_k(x) = \cos\left(k\cos^{-1}(x)\right)$. 
One way to approximate $g(x)$ is to truncate (\ref{eq:aChebSeries}) after the
first $n$ terms, to obtain the polynomial (of degree at most $n-1$)
\begin{equation}
g_{\text{trunc}}(x) = \sum_{k=0}^{n-1} g_kT_k(x).
\label{eq:truncate}
\end{equation}
The $n$ coefficients $\left\{g_k\right\}$ can be obtained by numerical quadrature in $\mathcal{O}(n^2)$ operations.  A second approach is to interpolate
$g(x)$ at $n$ Chebyshev points --- i.e., $\cos(k\pi/(n-1))$ for $k=0,1,\ldots,n-1$ --- to obtain
the polynomial
\begin{equation}
g_{\text{interp}}(x) = \sum_{k=0}^{n-1} \tilde{g}_kT_k(x).
\label{eq:interpolant}
\end{equation}
In this case, the $n$ coefficients $\left\{\tilde{g}_k\right\}$ can be computed with a Fast Cosine Transform in just $\mathcal{O}(n \log n)$ operations.
The coefficients $\left\{\tilde{g}_k\right\}$ and
$\left\{g_k\right\}$ are closely related by an aliasing formula
stated by Clenshaw and Curtis \cite{ClenshawCurtis}.  

Usually, in practice, $g(x)$ will be many times
differentiable on $[-1,1]$, and then (\ref{eq:truncate}) and
(\ref{eq:interpolant}) converge uniformly to $g(x)$ at an algebraic rate as 
$n\rightarrow\infty$.  Moreover, the convergence is spectral (super-algebraic)
when $g$ is infinitely differentiable, which improves to be exponential when $g$
is
analytic in a neighbourhood of $[-1,1]$.  If $g$ is entire, the convergence rate improves
further to be super-exponential.  

Mathematically, we seek the Chebyshev coefficients of the
truncation (\ref{eq:truncate}), of $a(x)$ and $f(x)$ which are defined by
the integrals (\ref{eq:aChebSeries}). However, we use the
coefficients in the polynomial interpolant (\ref{eq:interpolant}) as an
approximation, since the coefficients match to machine precision for sufficiently large $n$.   
The degree of the polynomial used to approximate $a(x)$ will later
be closely related to the bandwidth of the almost banded linear system associated to (\ref{eq:firstDE}).

The spectral method in this paper solves for the coefficients of the solution in
a Chebyshev series. In order to achieve this, we need to be
able to represent differentiation, $u'(x)$, and
multiplication, $a(x)u(x)$, in terms of operators on coefficients.
\subsection{First order differentiation operator}
The derivatives of Chebyshev polynomials satisfy 
\begin{equation}
\frac{dT_k}{\dx} = 
\begin{cases}
kC^{(1)}_{k-1}& k\geq 1\\
0&k=0\\
\end{cases}
\label{eq:diffrelation}
\end{equation}
where $C^{(1)}_{k-1}(x)$
is the Chebyshev polynomial of the second kind of degree $k-1$. We use
$C^{(1)}$ instead of $U$ to highlight the connection to ultraspherical
polynomials, which are key to extending the method to higher order differential
equations (see section \ref{sec:higherorder}).

Now, we use
(\ref{eq:diffrelation}) to derive a simple expression for the derivative of
a Chebyshev series.  Suppose that $u(x)$ is given by the
Chebyshev series 
\begin{equation}
u(x) = \sum_{k=0}^\infty u_kT_k(x).
\label{eq:uChebSeries}
\end{equation}
Differentiating scales the coefficients and changes the basis:
\[
u'(x) = \sum_{k=1}^\infty ku_kC^{(1)}_{k-1}(x).
\]
In other words, the vector of coefficients of the derivative in a $C^{(1)}$ series is given by
$\D_0\textbf{u}$
where $\D_0$ is the differentiation operator
\[
\D_0 = \begin{pmatrix}
					0 & 1 &  \cr
					&  & 2 &  \cr
					&&  & 3  \cr			
	
					&&& & \ddots  
			\end{pmatrix}
\]
and $\textbf{u}$ is the (infinite) vector of Chebyshev coefficients for $u(x)$.  Note that
this differentiation operator is sparse, in stark
contrast to the classic differentiation operator in spectral collocation methods
\cite{Boyd_01_01}.

\subsection{Multiplication operator for Chebyshev series}
In order to handle variable coefficients of the form $a(x)u(x)$ in
(\ref{eq:firstDE}), we need
to represent the 
multiplication of two Chebyshev series as an operator on
coefficients. To this end, we express $a(x)$ and $u(x)$ in terms of their Chebyshev series
\[
a(x) = \sum_{j=0}^\infty a_jT_j(x)\qqand u(x) = \sum_{k=0}^\infty u_k T_k(x)
\]
and multiply these two expressions together to obtain
\[
a(x)u(x) = \sum_{j= 0}^\infty\sum_{k=0}^\infty a_ju_kT_j(x)T_k(x) =
\sum_{k=0}^\infty
c_kT_k(x),
\]
desiring an explicit form for $\mathbf{c} = \left(c_0,c_1,\ldots,\right)^\top$.  By Proposition
$2.1$ of \cite{Baszenski_97_01} we have that
\begin{equation}
c_k = 
\begin{cases}
a_0u_0 + \frac{1}{2}\sum_{l=1}^\infty a_lu_l & k=0\vspace{4pt}\\
\frac{1}{2}\sum_{l=0}^{k-1}a_{k-l}u_l + a_0u_k + \frac{1}{2}\sum_{l=1}^\infty
a_lu_{l+k} + \frac{1}{2}\sum_{l=0}^{\infty}a_{l+k}u_l& k\geq1,\\
\end{cases}
\label{eq:Multrelation}
\end{equation}
and in terms of operators $\textbf{c}=\mathcal{M}_0[{a}]\textbf{u}$ where $\mathcal{M}_0[{a}]$ is a Toeplitz plus an almost Hankel operator given by
\[
\mathcal{M}_0[{a}] =
\frac{1}{2}
\left[
\begin{pmatrix}
2a_0& a_1 & a_2 & a_3 &\ldots\\
a_1 & 2a_0& a_1 & a_2 &\ddots\\
a_2 &a_1 & 2a_0& a_1 &\ddots\\
a_3 &a_2 &a_1 & 2a_0&\ddots\\
\vdots & \ddots & \ddots &\ddots&\ddots\\
\end{pmatrix}
+ 
\begin{pmatrix}
0& 0 & 0 & 0 &\ldots\\
a_1 & a_2& a_3 & a_4 &\iddots\\
a_2 &a_3 & a_4& a_5 &\iddots\\
a_3 &a_4 &a_5 & a_6&\iddots\\
\vdots & \iddots & \iddots &\iddots&\iddots\\
\end{pmatrix}
\right].
\]
At first glance, it appears that the multiplication operator and any truncation
of it are dense.  However,
since $a(x)$ is continuous with bounded variation, we are able to
uniformly approximate $a(x)$
with a finite number of Chebyshev coefficients to any desired accuracy. That
is, for any $\epsilon>0$ there exists an $m\in\mathbb{N}$ such that 
\[
\left\|a(x)-\sum_{k=0}^{m-1} a_k T_k(x)\right\|_{L_\infty([-1,1])}<\epsilon.
\]
As long as $m$ is large enough, to all practical purposes, we can use the
truncated Chebyshev series to replace $a(x)$. Recall, in our
implementation we approximate this truncation by the
polynomial interpolant of the form (\ref{eq:interpolant}). Hence,
the $n\times n$ principal part of $\M_0[{a}]$ is
banded with bandwidth $m$ for $n>m$. Moreover, $m$ can be surprisingly small when $a(x)$ is analytic or many
times differentiable.  

There is still one ingredient absent: The operator
$\mathcal{D}_0$ returns coefficients in a $C^{(1)}$ series,
whereas the operator $\mathcal{M}_0[{a}]$ returns coefficients in
a Chebyshev series; hence, $\mathcal{D}_0 + \mathcal{M}_0[{a}]$ is meaningless.  In order to correct this, we
require an operator that maps coefficients in a Chebyshev series to those in
a $C^{(1)}$ series. 

\subsection{Conversion operator for Chebyshev series}
The Chebyshev polynomials $T_k(x)$ can be written in
terms of the $C^{(1)}$ polynomials by the recurrence relation
\begin{equation}
T_k = 
\begin{cases}
\frac{1}{2}\left(C^{(1)}_k - C^{(1)}_{k-2}\right)& k\geq2 \vspace{2pt}\\
\frac{1}{2}C^{(1)}_1 & k=1\vspace{2pt}\\
C^{(1)}_0 & k=0.
\end{cases}
\label{eq:transformation}
\end{equation}
A more general form of (\ref{eq:transformation}), which we will use later, is
given in $\cite{NISTHandbook}$. Suppose that $u(x)$ is given by the
Chebyshev series (\ref{eq:uChebSeries}). Then, using (\ref{eq:transformation}) we
have
\[
\begin{aligned}
u(x) = \sum_{k=0}^\infty u_kT_k(x) &= u_0C^{(1)}_0(x) +
\frac{1}{2}u_1C^{(1)}_1(x) + \frac{1}{2}\sum_{k=2}^\infty u_k\left(C^{(1)}_k(x)
-
C^{(1)}_{k-2}(x)\right)\\
& = \left(u_0-\frac{1}{2}u_2\right)C^{(1)}_0(x) +
\sum_{k=1}^\infty\frac{1}{2}\left(u_{k}-u_{k+2}\right)C_k^{(1)}(x).
\end{aligned}
\]
Hence, the $C^{(1)}$ coefficients for $u(x)$ are $\S_0\textbf{u}$
where $\S_0$ is the conversion operator 
\[
\S_0 = \begin{pmatrix}
					1 &  & -\half \cr
					& \half &  & -\half \cr
					&& \half &  & -\half \cr		
		
					&&&\ddots &  & \ddots
			\end{pmatrix}
\]
and $\textbf{u}$ is the vector of Chebyshev coefficient for $u(x)$.
Note that this conversion operator, and any truncation of it, is sparse and
banded.
\subsection{Discretization of the system}
Now, we have all the ingredients to solve any differential equation of
the form (\ref{eq:firstDE}).  Firstly, we can represent the differential
operator as 
\[
\L : = \D_0 + \S_0\M_0[a],
\]
which takes coefficients in a Chebyshev series to those in a
$C^{(1)}$ series. Due to this fact, the right-hand side
$f(x)$ must be expressed in terms of its coefficients in a $C^{(1)}$ series.
Hence,  we can represent the differential
equation (\ref{eq:firstDE}), without its boundary conditions, as
\[
\L\mathbf{u} = \S_0\mathbf{f}.
\]
where $\mathbf{u}$ and $\mathbf{f}$ denote the vectors of
coefficients in the
Chebyshev series of the form (\ref{eq:aChebSeries}) for $u(x)$ and $f(x)$,
respectively. 

We truncate the operators to derive a practical numerical scheme, which corresponds to applying the 
$n \times \infty$ projection operator given by 
\begin{equation}
\P_n = (I_n, {\mathbf 0}),
\label{eq:projection}
\end{equation} 
where $I_n$ is the square identity matrix of size $n$. 
Truncating the
differentiation operator to
$\P_n\D_0\P_n^\top$ results in an $n\times n$ matrix with a zero last row.
This observation motivates us to impose the boundary condition by replacing
the last row of $\P_n\L \P_n^\top$.  We
take the convention of permuting this boundary row to the first row so that
 the linear system is close to upper triangular.  That is, in order to
obtain an
approximate solution to (\ref{eq:firstDE}) we solve the system
\begin{equation}
A_n
\left( 
\begin{array}{c}
u_0\\
u_1\\
\vdots\\
u_{n-1}
\end{array}
\right)
=
\left(
\begin{array}{c}
c\vspace{15pt}\\
\left(\P_{n-1}\S_0\P_n^\top\right)\left( \P_n\mathbf{f}\right)\\
\vspace{4pt}
\end{array}
\right),
\label{eq:system}
\end{equation}
with
	$$A_n = \left(
 \begin{array}{c c c c}
T_0(-1) & T_1(-1) & \ldots & T_{n-1}(-1)\\
&&&\\
 \multicolumn{4}{c}{\P_{n-1} \L \P_n^\top}  \\
&&&\\
 \end{array}
\right).$$
(Note that $T_k(-1) = (-1)^{k}$.)  The solution $u(x)$ is
then approximated by the computed solution,
\[
\tilde{u}(x)=\sum_{k=0}^{n-1} u_kT_k(x).
\]

\begin{remark} \hbox{ }
Traditionally, one   generates the spectral system by
discretizing each operator individually; i.e., we would make the approximation
$$\P_{n-1}\L\P_n^\top \approx\P_{n-1}\D_{0}\P_{n}^\top +
\left(\P_{n-1}\S_0\P_n^\top\right) \left(\P_n\M[a]\P_n^\top\right).$$
However, we use can generate $\P_{n-1}\L\P_n^\top$ precisely, via
	$$\P_{n-1}\L\P_n^\top  = \P_{n-1}\D_{n}\P_{n}^\top +
\left(\P_{n-1}\S_0\P_{n+m}^\top\right) \left(\P_{n+m}\M[a]\P_n^\top\right).$$
This means that each row of the operator can be generated exactly and allows us to develop the adaptive QR algorithm in section \ref{sec:fastsolver}.  Moreover, exact truncations are of critical importance for infinite dimensional linear algebra and for approximating spectra and pseudospectra \cite{Hansen_08}.  
\end{remark}
 
\subsection{Numerical examples}\label{sec:firstorderexamples}
The traditional approach solves the linear system (\ref{eq:system}) for
progressively larger $n$, and terminates the process when the tail of the
solution falls below relative magnitude of machine epsilon.  This is the procedure employed in Chebfun \cite{Chebfun,Driscoll_08_01} and results in $\O(n_{\rm opt} \log n_{\rm opt})$ complexity.  Instead the adaptive QR algorithm of section \ref{sec:fastsolver} can be used to find the optimal truncation and solve for the solution, concurrently which reduces the complexity to just $\O(n_{\rm opt})$.  

For the first example, we consider the linear differential equation 
\begin{equation}
u'(x) + x^3u(x) = 100\sin\left(20,\!000x^2\right)\qquad u(-1)=0
\label{eq:firstexample}
\end{equation}
which has a highly oscillatory forcing term. The exact solution is
\[
u(x) = e^{-\frac{x^4}{4}}\left(\int_{-1}^x
100e^{\frac{t^4}{4}}\sin\left(20,\!000t^2\right)dt\right).
\]
\begin{figure}
  \centering
    \includegraphics[width=\textwidth]{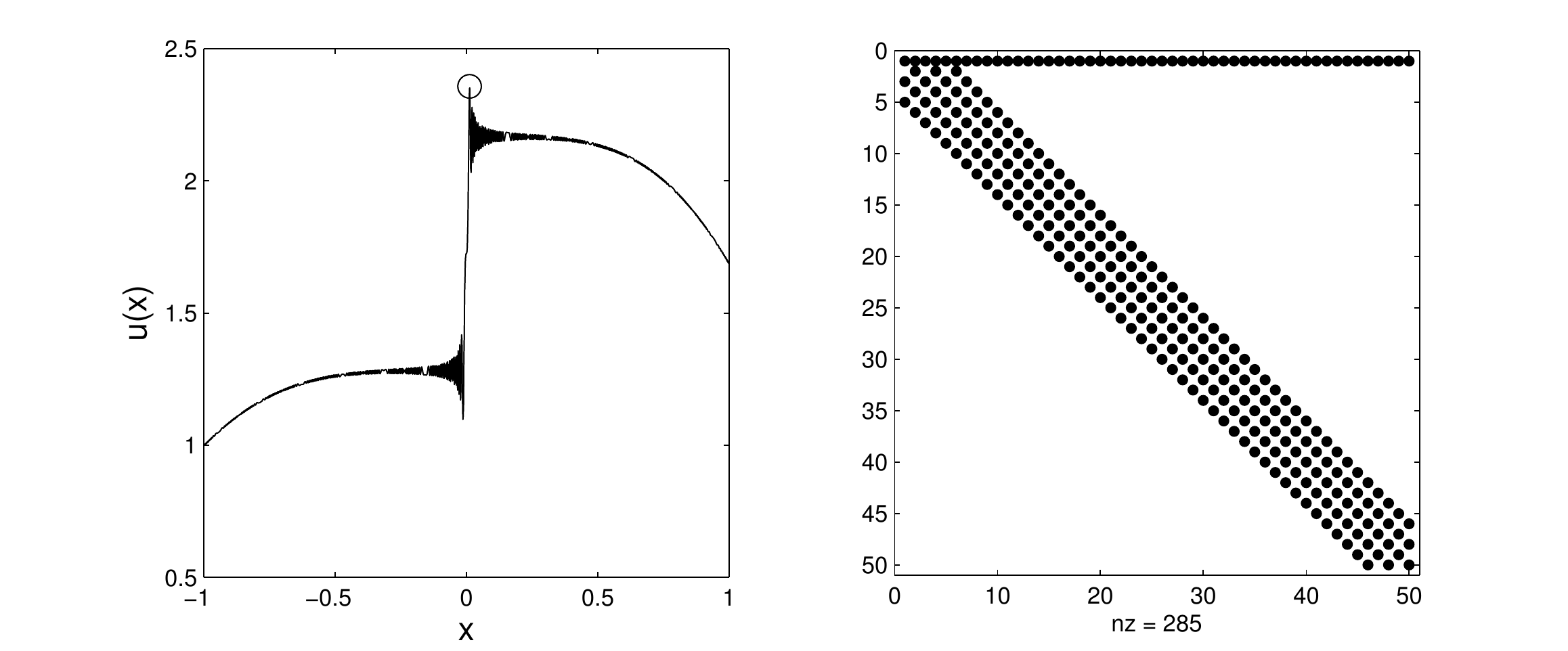}
\caption{Left: The highly oscillatory solution to
(\ref{eq:firstexample}) with the maximum of the solution computed and
marked by a circle. Right: Sparsity structure of the almost banded linear system, with the number of non-zero entries denoted by {\em nz}.}
\label{fig:largeRHS}
\end{figure}

The computed solution is a polynomial of degree $20,\!391$ and 
uniformly approximates the exact solution to essentially machine precision. The
computed solution is of very high degree because at least $2$
coefficients are required per oscillation --- the {\em Nyquist rate}. 
In Figure \ref{fig:largeRHS} we plot the computed oscillatory solution, and a
realization of the matrix when $n=50$, which shows the
almost banded structure of the linear system.

The approximate solution is expressed in terms of a Chebyshev basis, which is
convenient for further manipulation. For example, its maximum
is $2.3573$ (circled in Figure \ref{fig:largeRHS}), its integral is $3.2879$
and the equation $u(x)-1.3=0$ has $113$ solutions in $x\in[-1,1]$.

As a second example we consider the linear differential equation
\begin{equation}
u'(x) + \frac{1}{ax^2+1}u(x) = 0 \qqand u(-1)=1.
\label{eq:secondexample}
\end{equation}
We take $a = 5\times 10^4$, in which case the variable coefficient
can be approximated to roughly machine
precision by a polynomial of degree $7,\!350$, and hence, only for very large
$n$ is the linear system (\ref{eq:system}) banded.
\begin{figure}
  \centering
    \includegraphics[width=\textwidth]{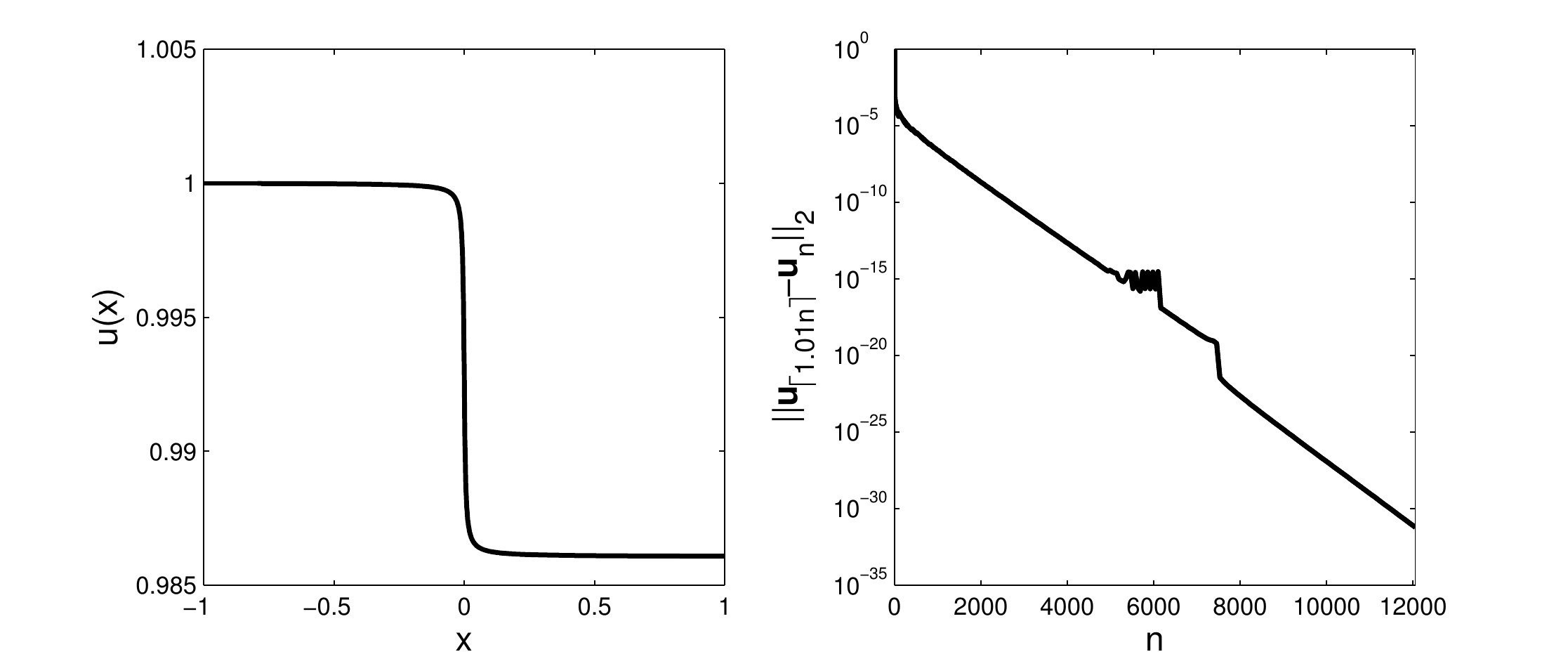}
\caption{Left: The computed solution to
(\ref{eq:secondexample}) with $a=5\times 10^4$. Right: Plot of the Cauchy error
for the
solution coefficients, which shows the 2-norm difference between the
coefficients of the approximate solution when solving an $n\times n$ and
an $\lceil 1.01n\rceil \times \lceil 1.01n\rceil$ matrix system.
}  \label{fig:largevariablecoeffs}
\end{figure}
The exact solution to (\ref{eq:secondexample}) is
\[
u(x) = 
\exp\left(-\frac{\tan^{-1}(\sqrt{a}x) + \tan^{-1}(\sqrt{a})}{\sqrt{a}}\right),
\]
which can be approximated to machine precision by a polynomial of
degree $5,\!377$, determined by interpolating at Chebyshev points.  On the other hand, the computed solution $\tilde{u}(x)$, is a polynomial
of degree $5,\!093$ such that  
\[
\left(\int_{-1}^{1} \left(u(x) - \tilde{u}(x)\right)^2\dx\right)^{\frac{1}{2}} = 
 2.86\times 10^{-15}.
\]
The solution contains a thin boundary layer and contains a singularity in the complex plane 
which is close in proximity to the $[-1,1]$. This means 
that the associated Berstein ellipse is restricted and therefore, a large degree polynomial is required 
to approximate the solution (see, for example, \cite{Davis_75_01}). 
A plot of the solution and the Cauchy error  are included
in Figure \ref{fig:largevariablecoeffs}.  The Cauchy error plot confirms that the
solution is, up to machine precision, independent of the number of coefficients
in its expansion for $n\geq 5,\!100$.

\section{Ultraspherical polynomials and high order differential equations}\label{sec:higherorder}

We now generalize the approach to high order differential equations of the
form
\begin{equation}
\sum_{\lambda = 0}^N a^{\lambda}(x)\frac{\dkf^{\lambda}u(x)}{\dx^\lambda} = f(x)
\hspace{1cm} \text{on } [-1,1],
\label{eq:higherDE}
\end{equation}
with general boundary conditions $\B u = \mathbf c$.    We assume that the boundary operator $\B$
is given in terms of the Chebyshev coefficients of $u$.  For example, for Dirichlet conditions
	$$\B = \begin{pmatrix} T_0(-1) & T_1(-1) & T_2(-1) & \cdots \\
					T_0(1) & T_1(1) & T_2(1) & \cdots\end{pmatrix} = \begin{pmatrix} 1 & -1 & 1 & -1 & \cdots \\
					1 & 1 & 1 & 1 & \cdots\end{pmatrix},$$
and for Neumann conditions
	$$\B = \begin{pmatrix} T_0'(-1) & T_1'(-1) & T_2'(-1) & \cdots \\
					T_0'(1) & T_1'(1) & T_2'(1) & \cdots\end{pmatrix}  = \begin{pmatrix} 0 & 1 & -4 & \cdots & (-1)^{k+1} k^2& \cdots \\
					0 & 1 & 4 & \cdots & k^2 & \cdots \end{pmatrix}.$$
We can also impose less standard boundary conditions;  e.g., we can impose that the solution integrates to a constant by using the Clenshaw--Curtis weights \cite{ClenshawCurtis} (computable in $\O(n \log n)$ operations \cite{FastClenshawCurtisWeights}) or it evaluates to a fixed constant in $(-1,1)$.  In fact, because we are using boundary bordering, any boundary condition which depends linearly on the solution's coefficients can be imposed in an automated manner.

%We continue to impose that all the variable
%coefficients are continuous with bounded variation to ensure that each
%coefficient can be
%represented by a finite Chebyshev series.  

The approach of the first order method relied on three relations:
differentiation (\ref{eq:diffrelation}), multiplication (\ref{eq:Multrelation}) and conversion
(\ref{eq:transformation}).
To generalize the spectral method to higher order
differential equations we use similar relations, now in terms of higher order
ultraspherical polynomials. 
	
The ultraspherical (or Gegenbauer) polynomials
$C_0^{(\lambda)}(x),C_1^{(\lambda)}(x),\ldots$ are a family of polynomials
orthogonal with respect to the weight
\[
(1-x^2)^{\lambda-\frac{1}{2}}.
\]
We will only use ultraspherical polynomials for $\lambda=1,2,\ldots$, defined
uniquely by  normalizing the leading coefficient so that
	$$C_k^{(\lambda)}(x) = \frac{2^k (\lambda)_k}{k!} x^k + \O(x^{k-1}),$$ 
where
$(\lambda)_k = \frac{(\lambda+k-1)!} {(\lambda-1)!}$ denotes the
\emph{Pochhammer
symbol}.  In particular, the ultraspherical polynomials with $\lambda=1$ are the
Chebyshev polynomials of the second kind, which we denote by $C^{(1)}$. 

Importantly, ultraspherical polynomials satisfy
an analogue of (\ref{eq:diffrelation}) \cite{NISTHandbook},
\begin{equation}
\frac{dC_k^{(\lambda)}}{\dx} = 
\begin{cases}
2 \lambda C_{k-1}^{(\lambda+1)} & k \geq1\\
0 & k=0.\\ 
\end{cases}, \qquad \lambda \geq 1.
\label{eq:higherdiff}
\end{equation}
Moreover, they also satisfy an analogue to (\ref{eq:transformation})
\begin{equation}
C_k^{(\lambda)} = 
\begin{cases}
\frac{\lambda}{\lambda + k} \left(C_k^{(\lambda+1)} -
C_{k-2}^{(\lambda+1)}\right) & k\geq2\\
\frac{\lambda}{\lambda + 1} C_1^{(\lambda + 1)} & k=1\\
C_0^{(\lambda + 1)} & k=0,\\
\end{cases} \qquad \lambda\geq1.
\label{eq:highertransformation}
\end{equation}

Suppose that $u(x)$ is represented as the Chebyshev series
(\ref{eq:uChebSeries}). Then, for $\lambda = 1, 2,\ldots$, 
(\ref{eq:diffrelation}) implies that
\[
\frac{\dkf^\lambda u(x)}{\dx^\lambda} = \sum_{k=1}^\infty
ku_k\frac{\dkf^{\lambda-1} C^{(1)}_{k-1}(x)}{\dx^{\lambda-1}}.
\]
By applying the relation (\ref{eq:higherdiff}) $\lambda-1$ times we obtain
\[
\frac{\dkf^\lambda u(x)}{\dx^\lambda} =
2^{\lambda-1}(\lambda-1)!\sum_{k=\lambda}^\infty
k u_k C^{(\lambda)}_{k-\lambda}(x).
\]
This means that the $\lambda$-order differentiation
operator  takes the form
\[
\mathcal{D}_\lambda = 
2^{\lambda-1} (\lambda-1)!
\begin{pmatrix}
\overbrace{ 0\quad\cdots\quad0}^{\mbox{$\lambda$ times}}&\lambda&&&\\
& &\lambda+1&\\
& & &\lambda+2&\\
   & & & &\ddots\\
\end{pmatrix}.
\]
The $\lambda$-order differentiation operator that does not change bases can be constructed from recurrence relations \cite{Doha_04_01} and 
this can be used to construct spectral methods \cite{Doha_02_01}. However, the differentiation operator is not sparse.

In the process of differentiation $\D_\lambda$
converts coefficients in a Chebyshev series to coefficients in a
$C^{(\lambda)}$ series. Moreover, using
(\ref{eq:highertransformation}) the operator which converts
coefficients in $C^{(\lambda)}$ series to those in a
$C^{(\lambda+1)}$ series is given by
\[
\S_\lambda = 
\begin{pmatrix}
1&            & -\frac{\lambda}{\lambda+2} &               & \\
 & \frac{\lambda}{\lambda+1}&              &  -\frac{\lambda}{\lambda+3} &\\
 &            & \frac{\lambda}{\lambda+2}  &         & 
-\frac{\lambda}{\lambda+4}&\\
&&            & \ddots  &         &  \ddots \\
\end{pmatrix}.
\]

As before, we also require a multiplication operator, but this time representing the multiplication between two
$C^{(\lambda)}$ ultraspherical
series.

\subsection{Multiplication operator for ultraspherical series}

In order to handle the variable coefficients in (\ref{eq:higherDE}), we must
represent multiplication of two
ultraspherical series in coefficient space. Given
two functions
\[
a(x) = \sum_{j= 0}^\infty a_j C^{(\lambda)}_j(x) \qqand u(x) =
\sum_{k= 0}^\infty u_k C^{(\lambda)}_k(x),
\]
we have
\begin{equation}
a(x)u(x) = \sum_{j= 0}^\infty\sum_{k= 0}^\infty a_ju_k
C^{(\lambda)}_j(x)C^{(\lambda)}_k(x).
\label{eq:mult}
\end{equation}
To obtain a $C^{(\lambda)}$ series for $a(x)u(x)$ we use the linearization
formula given by Carlitz \cite{Carlitz_61_01}, which takes the form
\begin{equation}
C_j^{(\lambda)}(x)C_k^{(\lambda)}(x) = \sum_{s=0}^{\min(j,k)}
c^\lambda_s(j,k)C_{j+k-2s}^{(\lambda)}(x)
\label{eq:linearization}
\end{equation}
where 
\begin{equation}
c^\lambda_s(j,k) = \frac{j+k+\lambda-2s}{j+k+\lambda
-s}\frac{(\lambda)_s(\lambda)_{j-s}(\lambda)_{k-s}}{s!(j-s)!(k-s)!}\frac{
(2\lambda)_{j+k-s} } {(\lambda)_{j+k-s} }\frac{(j+k-2s)!}{(2\lambda)_{j+k-2s}}
\label{eq:cs}
\end{equation}
and $(\lambda)_{k} = \frac{(\lambda+k-1)!}{(\lambda-1)!}$ is the Pochhammer
symbol. We
substitute (\ref{eq:linearization}) into (\ref{eq:mult}) and rearrange the
summation signs to obtain 
\begin{equation}
a(x)u(x) = \sum_{j= 0}^\infty \left(\sum_{k=0}^\infty\sum_{s=\max(0,k-j)}^k
a_{2s+j-k}c^\lambda_{s}(k,2s+j-k) u_k \right) C_j^{(\lambda)}(x).
\label{eq:Multform}
\end{equation}
From (\ref{eq:Multform}) the $(j,k)$ entry of the multiplication operator
representing the product of
$a(x)$ in a $C^{(\lambda)}$ series is
\[
\M_\lambda[a]_{j,k} =\sum_{s=\max(0,k-j)}^k
a_{2s+j-k}c^\lambda_{s}(k,2s+j-k),
\hspace{1cm} j,k\geq 0.
\]
In practice, $a(x)$ will be approximated by a truncation of its $C^{(\lambda)}$ series,
\begin{equation}\label{eq:Clambdaseries}
a(x) \approx \sum_{j=0}^{m-1} a_jC^{(\lambda)}_j(x),
\end{equation}
and with this approximation the matrix $\P_n\M_\lambda[a]\P_n^\top$ is banded with
bandwidth $m$ for $n>m$.  The expansion \eqref{eq:Clambdaseries} can be computed
by approximating the first $m$ Chebyshev coefficients in the Chebyshev series
for $a(x)$ and then applying a truncation of the conversion operator
$\S_{\lambda-1} \cdots S_0$.

The formula for $c_s^\lambda(j,k)$, (\ref{eq:cs}),  cannot be used
directly to form $\M_\lambda[a]$ due to arithmetic overflow
issues that arise for $j, k \geq 70$. Instead, we cancel terms
in the numerator and denominator of (\ref{eq:cs}) and match up the remaining
terms of similar magnitude to obtain an equivalent, but more numerically stable formula
\begin{eqnarray}
c^\lambda_s(j,k) &= &\frac{j+k+\lambda-2s}{j+k+\lambda
-s}\times \prod_{t=0}^{s-1} \frac{\lambda+t}{1+t}  \times
\prod_{t=0}^{j-s-1} \frac{\lambda+t}{1+t} \nonumber\\&\times&
\prod_{t=0}^{s-1} \frac{2\lambda+j+k-2s+t}{\lambda+j+k-2s+t}\times
\prod_{t=0}^{j-s-1} \frac{k-s+1+t}{k-s+\lambda+t}.
\label{eq:myproduct}
\end{eqnarray}
All the fractions are of magnitude $\mathcal{O}(1)$ in size, and hence
$\P_n\M_\lambda[a]\P_n^\top$ can be
formed in floating point arithmetic.  For the purposes of computational speed, we
only apply (\ref{eq:myproduct}) once per entry, and use the recurrence relation
\begin{eqnarray}
c^\lambda_{s+1}(j,k+2) = c^\lambda_{s}(j,k)&\times&
\frac{j+k+\lambda-s}{j+k+\lambda-s+1} \times
\frac{\lambda+s}{s+1}\times\nonumber\\
&\times&\frac{j-s}{\lambda+j-s-1}\times\frac{2\lambda+j+k-s}{\lambda+j+k-s}
\times\frac{k-s+\lambda}{k-s+1}\nonumber
\end{eqnarray}
to generate all the other terms required.

\begin{remark} \hbox{ }In the special case when $\lambda=1$, 
the multiplication operator $\M_1[a]$ can be decomposed as a Toeplitz
operator plus a Hankel operator.\end{remark}

\subsection{Discretization of the system}
We now have everything in place to be able to solve high order differential
equations of the form (\ref{eq:higherDE}). Firstly, we can represent the
differential operator as 
\[
\L : = \M_N[a^N]\D_N + \sum_{\lambda=1}^{N-1}
\S_{N-1}\cdots\S_\lambda\M_\lambda[a^\lambda]\D_\lambda +
\S_{N-1}\cdots\S_0\M_0[a^0],
\]
which takes coefficients in a Chebyshev series to those in a
$C^{(N)}$ series. Due to this fact, the right hand side
$f(x)$ must be expressed in terms of its coefficients in a $C^{(N)}$
series.  Moreover, 
we impose the $K$ boundary conditions on the solution by replacing the last $K$
rows of
$\P_n\L \P_n^\top$ and permute these to the first $K$ rows.  That is,
in order to obtain an
approximate solution to (\ref{eq:higherDE}) we solve the system
\[
A_n\left(\begin{array}{c}
u_0\\
u_1\\
\vdots\\
u_{n-1}
\end{array}\right)
=
\left(\begin{array}{c}
\mathbf{c}\vspace{15pt}\\
\P_{n-K}\S_{N-1}\cdots \S_0\mathbf{f}\\
\vspace{4pt}
\end{array}\right),
\]
where
	$$A_n = \begin{pmatrix}
\mathcal{B}\P_n^\top \\
\\
\P_{n-K}\L \P_n^\top \\
\\
\end{pmatrix}
$$
and $\mathbf f$ is again a vector containing the Chebyshev coefficients of the
right-hand side $f$.  The solution $u(x)$ is then approximated by the $n$-term
Chebyshev series:
\[
u(x) \approx \sum_{k=0}^{n-1} u_kT_k(x).
\]
\subsection{Numerical examples}
For the first example we consider the Airy differential equation 
\begin{equation}
\epsilon u''(x) - xu(x) = 0 \qquad \text{with} \qquad 
u(-1)=\Ai\left(-\sqrt[3]{\frac{1}{\epsilon}}\right), \text{ }
u(1)=\Ai\left(\sqrt[3]{\frac{1}{\epsilon}}\right)
\label{eq:airyeqn}
\end{equation}
where $\Ai(\cdot)$ is the Airy function of the first kind. 
\begin{figure}
\begin{minipage}[b]{.5\textwidth}\centering
\includegraphics[width=\textwidth]{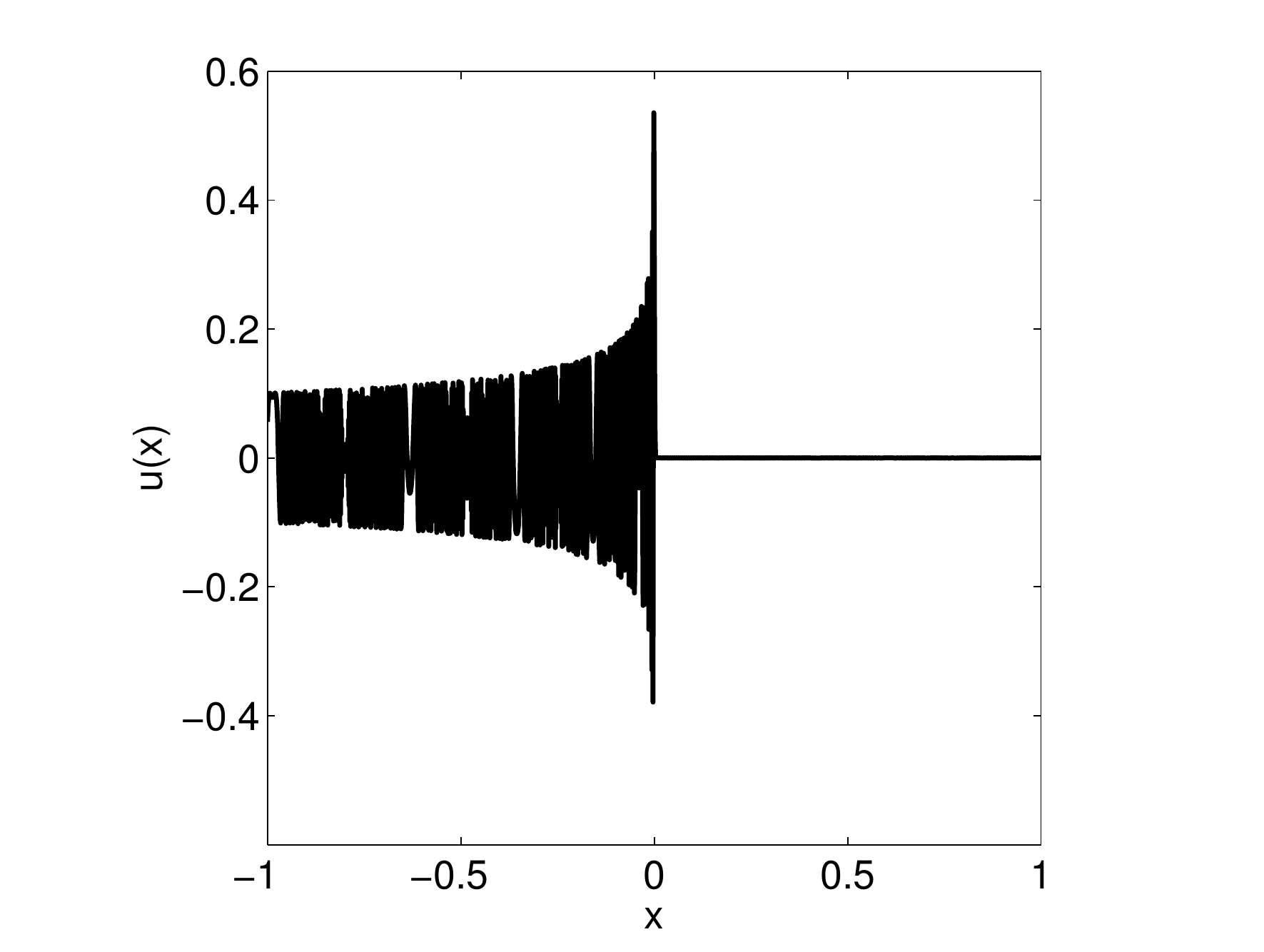}
\end{minipage}
\begin{minipage}[b]{.5\textwidth}\centering
\includegraphics[width=\textwidth]{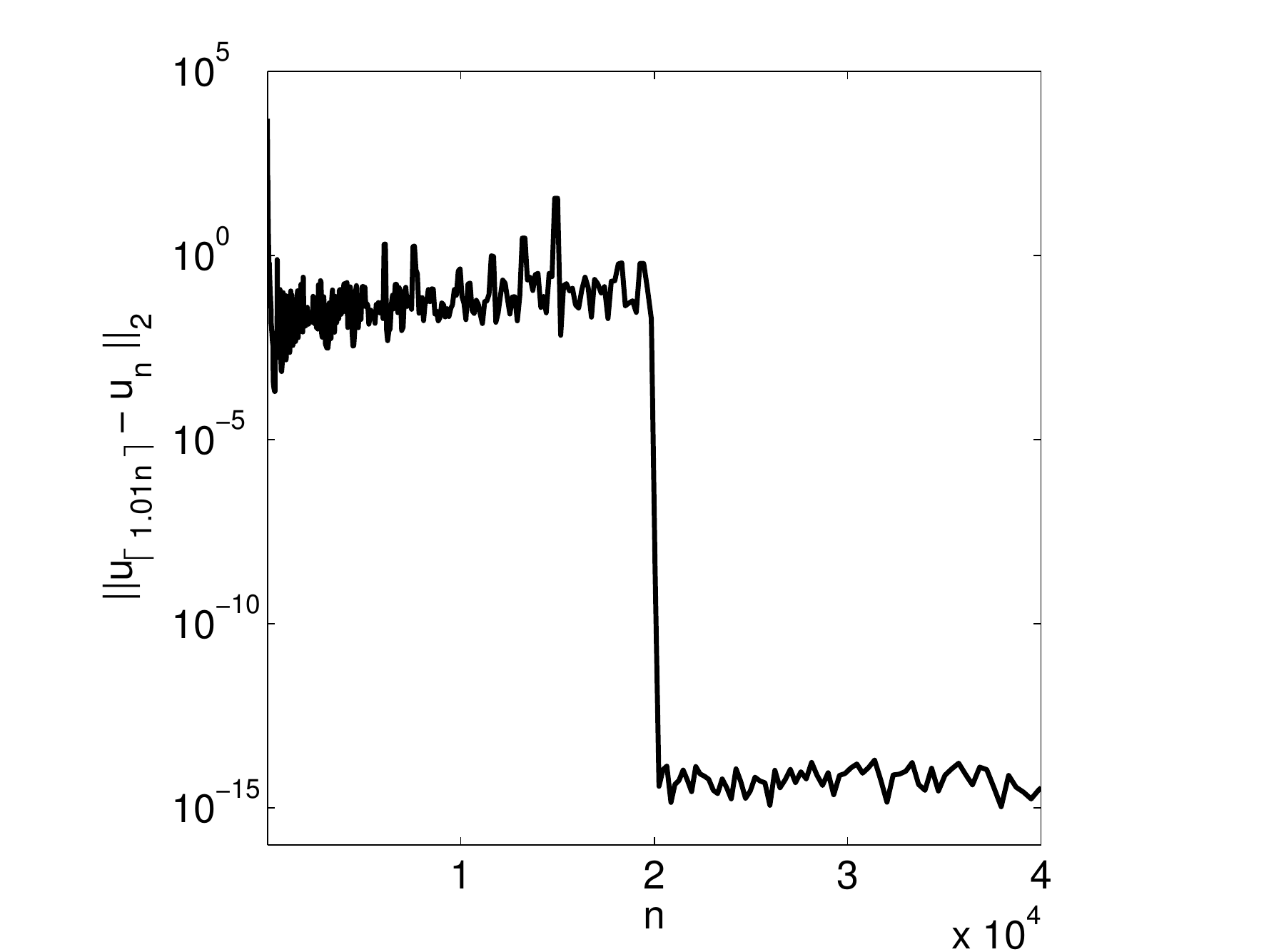}
\end{minipage}
\caption{Left: The highly oscillatory solution to
(\ref{eq:airyeqn}) with $\epsilon = 10^{-9}$. Right: The Cauchy error
for the
solution coefficients. The plot shows the 2-norm difference between the
coefficients of the approximate solution when solving an $n\times n$ and
$\lceil 1.01n\rceil \times \lceil 1.01n\rceil$ matrix system.} 
\label{fig:airyeqn}
\end{figure}

\begin{figure}
  \centering
\includegraphics[width=.5\textwidth]{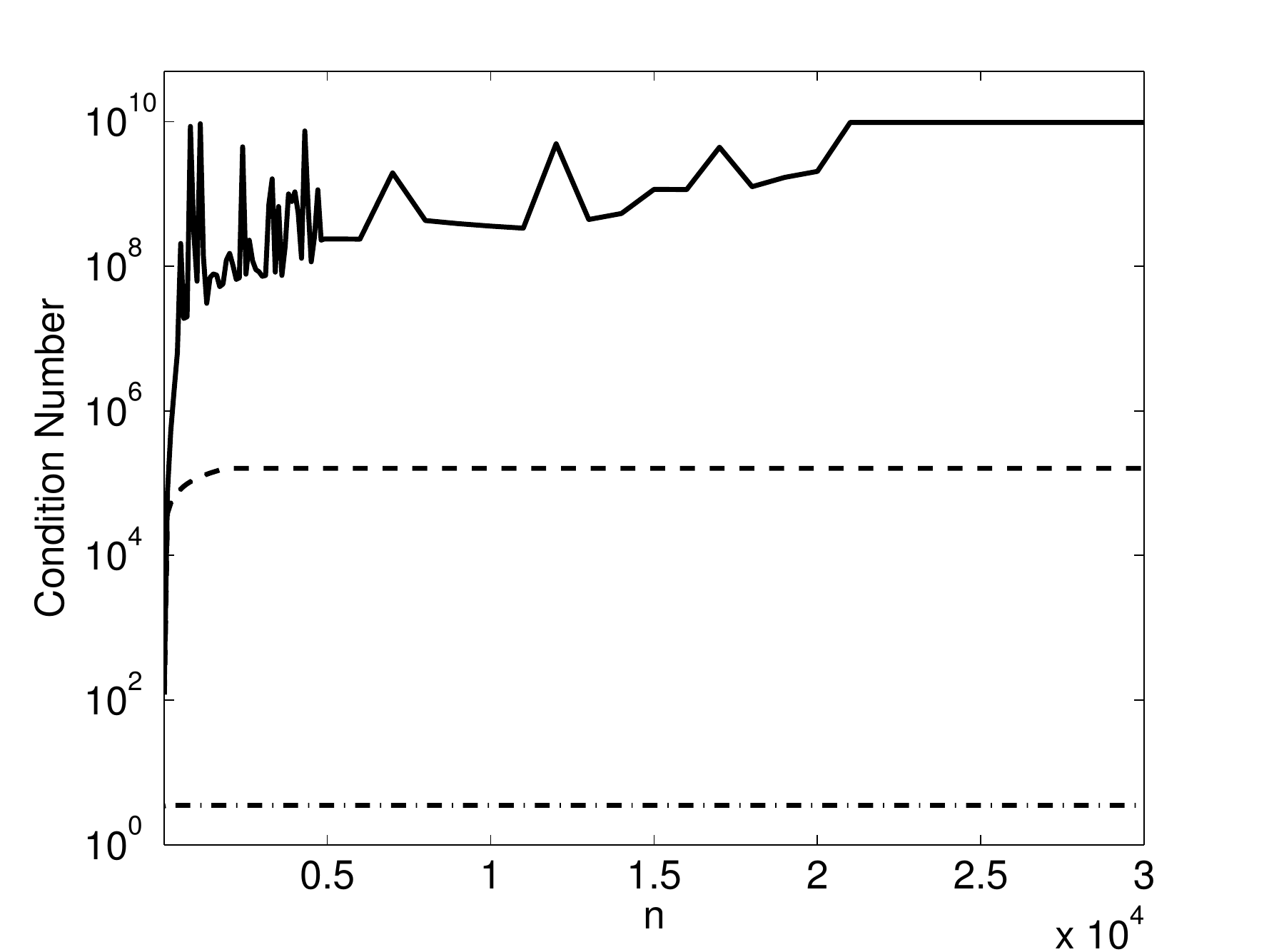}
\caption{Plot of the condition number of the linear systems  to solve (\ref{eq:airyeqn}) against the size of the discretization, for $\epsilon=1\times
10^{-9}$ (solid), 
$\epsilon=1\times 10^{-4}$ (dashed) and $\epsilon=1$ (dot-dashed).  For $\epsilon=1$ we employ the diagonal preconditioner presented in section \ref{sec:stable}.     The plot demonstrates
that the $2$-norm condition number is bounded from above. The observed error
in the solution is considerably better than what is suggested by the bounding constants, as seen in Fig.~\ref{fig:airyeqn}.} 
\label{fig:airycond}
\end{figure}

In Figure \ref{fig:airyeqn} we take $\epsilon=10^{-9}$ and plot the
computed solution which is a polynomial of degree $20,\!003$.  The exact solution
to (\ref{eq:airyeqn}) is the scaled Airy function,
\[
u(x) = \Ai\left(\sqrt[3]{\frac{1}{\epsilon}}x\right).
\]
Letting $\tilde{u}(x)$ denote the computed solution, we have
\[
\left(\int_{-1}^1 \left(u(x)-\tilde{u}(x)\right)^2\right)^{1/2} =
2.44\times 10^{-12}
\]
which is surprisingly good when compared with the ill-conditioning inherent in this singularly perturbed differential equation. Numerically we witness that the spectral method delivers much better accuracy than standard bounds based on the condition number would suggest.  
The Cauchy error plot in Figure \ref{fig:airyeqn} indicates that
the solution coefficients themselves are resolved to essentially machine precision, for $n\geq
20,\!000$. In Figure \ref{fig:airycond} we show numerical evidence that with a simple diagonal preconditioner, which we analysis in the next section the 
condition number of the linear systems formed are bounded for all $n$. Later,
for
$\epsilon=1$, we also show in Figure \ref{fig:higherordernorms} that the derivatives
of the solution are well approximated.

For the second example we consider the boundary layer problem,
\begin{equation}
\epsilon u''(x) - 2x\left(\cos(x)-\frac{8}{10}\right)u'(x) +
\left(\cos(x)-\frac{8}{10}\right)u(x) = 0
\label{eq:boundarylayer}
\end{equation}
with boundary conditions
\[
u(-1)=u(1)=1.
\]
Perturbation theory shows that the solution has two boundary layers at
$\pm \cos^{-1}(8/10)$ both of width $\O(\epsilon^{1/4})$. 
In Figure \ref{fig:boundarylayer} we take $\epsilon=10^{-7}$. The
computed solution is of degree $15,\!394$ and it is confirmed by the
Cauchy error plot that the solution is well-resolved. 
\begin{figure}
  \centering
    \includegraphics[width=\textwidth]{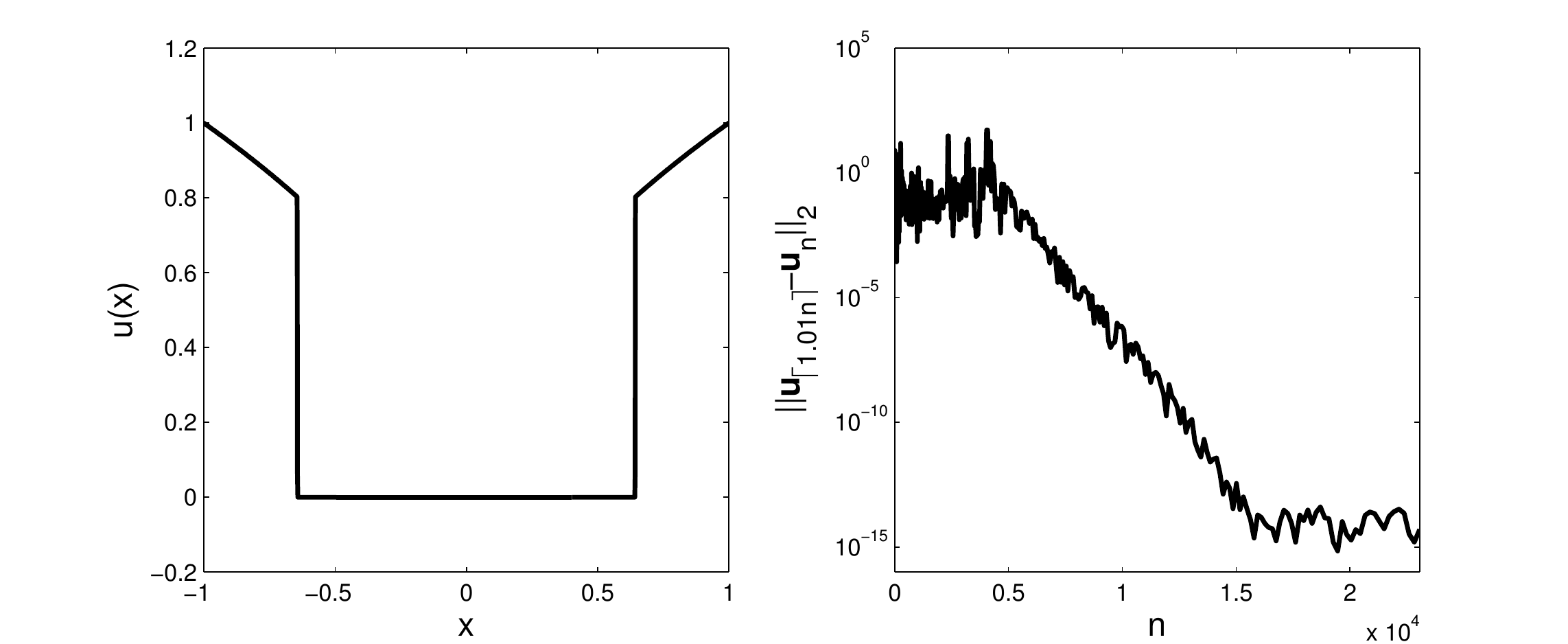}
\caption{Left: The solution to the boundary
layer problem (\ref{eq:boundarylayer}) with $\epsilon = 10^{-7}$. Right: The Cauchy error in
the $2$-norm for the solution coefficients.} 
\label{fig:boundarylayer}
\end{figure}

For the last example we consider the high order differential equation 
\[
u^{(10)}(x) + \cosh(x)u^{(8)}(x) + x^2u^{(6)}(x) + x^4u^{(4)}(x) +
\cos(x)u^{(2)}(x) + x^2u(x) = 0 
\]
with boundary conditions
\[
u(-1)=u(1)=0, \text{ } u'(-1)=u'(1)=1, \text{ } u^{(k)}(\pm1)=0,\text{
}2\leq k\leq 4.
\]
This is far from a practical example, and an exact solution seems difficult to 
construct. Instead, we note that if $u(x)$ is the solution
then it is odd; that is, $u(x) = -u(-x)$. Our method does not impose such a
condition and therefore, we can use it along with the Cauchy error to gain
confidence in the computed solution. The computed solution $\tilde{u}(x)$ is of
degree $55$ and plotted in Figure \ref{fig:higherorderexample}. Moreover, the
computed solution is odd to about machine precision,
\[
\left(\int_{-1}^1 \left(\tilde{u}(x)+\tilde{u}(-x)\right)^2\right)^{1/2} =
 1.252 \times 10^{-14}.
\]

\begin{figure}
  \centering
    \includegraphics[width=\textwidth]{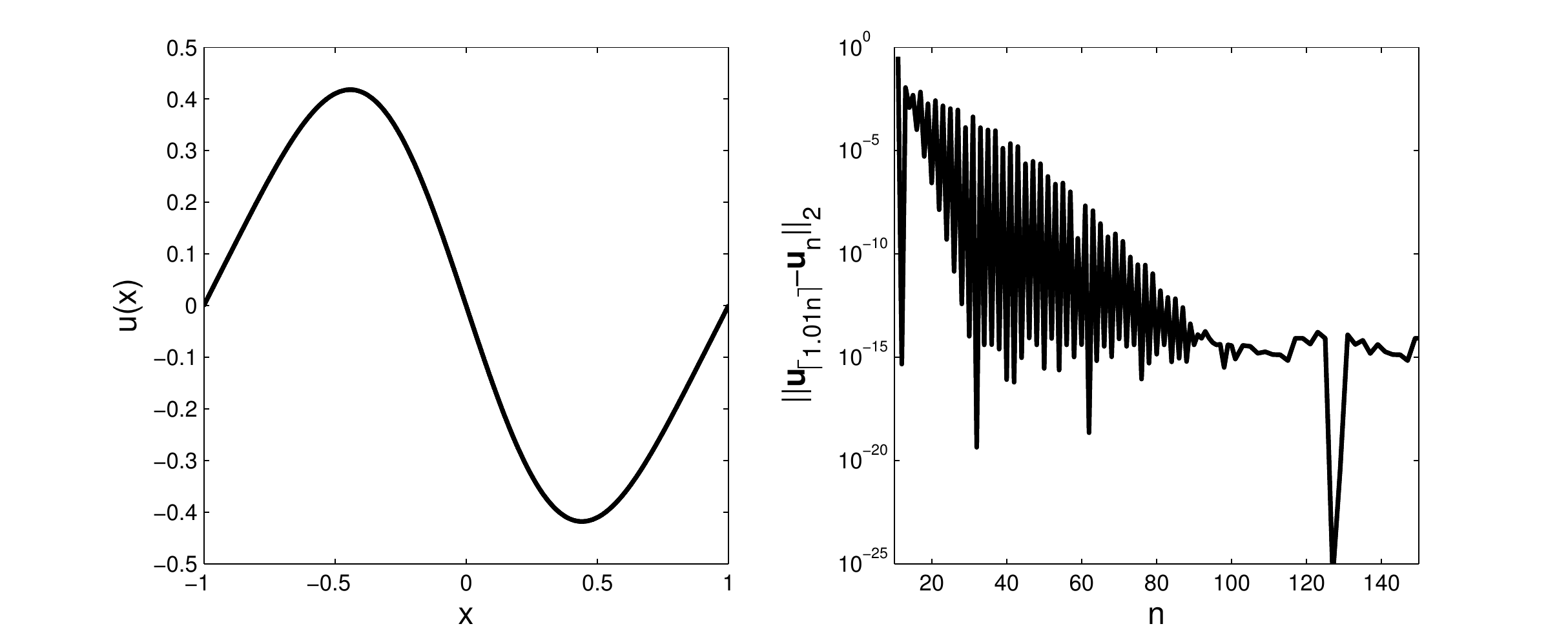}
\caption{Left: Plot of the solution to the $10$th order differential equation.
Right: Plot of the Cauchy error in
the $2$-norm for the solution coefficients.} 
\label{fig:higherorderexample}
\end{figure}

%Throughout these numerical experiments (and many others) we frequently see the
%coefficients of the solution decaying to significantly below machine precision.
%Moreover, the coefficients below machine precision are not noise, and can
%increase the accuracy of the solution. We do not, as of yet, completely
%understand this.

% s
% formulated in \cite{Trefethen_00_01}. It is a fourth order generalized
% eigenvalue problem given by 
% \[
% \frac{1}{R}\left(u''''(x) - 2u''(x) + u\right) - 2iu - i(1-x^2)\left(u''(x) -
% u(x)\right) = \lambda\left(u''(x) - u(x)\right) 
% \]
% with \emph{clamp} boundary conditions $u(\pm 1) = u'(\pm 1) = 0$. We
% truncate the operators and solve the almost banded generalised eigenvalue
% problem at the critical Renolds number $R=5772.22$. In Figure
% \ref{fig:orrsommerfeld} we plot the computed complex eigenvalues, using {\tt
% eigs} in {\sc Matlab}.
% 
% 
% \begin{figure}
%   \centering
%     \includegraphics[width=.8\textwidth]{orrsommerfeld}
% \caption{Plot of the $50$ rightmost eigenvalues of the Orr-Sommerfeld problem
% with the critical Reynolds number of $5772.22$.} 
% \label{fig:orrsommerfeld}
% \end{figure}
\section{Stability and convergence}\label{sec:preconditioner}\label{sec:stable}

The $2$-norm condition number of a matrix $A\in\mathbb{C}^{n\times n}$ is defined as
\[
\kappa(A) = \left\|A\right\|_2 \left\|A^{-1}\right\|_2.
\]

Without any preconditioning $\kappa\!\begin{pmatrix} A_n \end{pmatrix}$ grows proportionally with $n$, which
is significantly better than the typical growth of $\O(n^{2N})$ in the
condition number for the standard tau and collocation methods (see section
$4.3$ of \cite{Canuto_06_01}). However, the accuracy seen in practice even
outperforms this: the backward error is consistent with a numerical method with 
bounded condition number.  Later, we will show that a trivial,
diagonal preconditioner that scales the columns results in a linear system with a bounded condition
number.  However, we observe that even without preconditioning the linear systems can be solved to the same accuracy.  We explain this with the following proposition that the stability of QR is not affected by column scaling. 

%\sotodoinline{QR?}

\begin{proposition}
Suppose $R$ is a diagonal matrix. Solving $A \mathbf c =
\mathbf b$ using QR (with Givens rotations) is stable if QR applied to solve $A R \mathbf q = \mathbf
b$ is stable and $\|R\|_\infty \leq 1$.
\end{proposition}
\begin{proof} 
	This follows immediately from the invariance of Givens rotations to column scaling and the stability of back substitution, see \cite[pp.~374]{Higham_02_01}.
	\end{proof}

\subsection{A diagonal preconditioner and compactness}  Through-out this
section we assume that $a^N(x) = 1$ (otherwise, divide through by the
coefficient on the highest order term, assuming it is nonsingular). We make the restriction that $K=N$;
that is, the $N$th order differential equation has exactly $N$ boundary
conditions. When $K> N$ it is more appropriate to choose a non-diagonal
preconditioner, but we do not analyse that situation here. 

We
show that there exists a diagonal preconditioner so that the preconditioned
system has bounded condition number in high order norms (Definition
\ref{def:higherordernorms}). For Dirichlet boundary conditions the
preconditioned system has bounded condition number in the $2$-norm. 

Define the diagonal preconditioner by,
\[
\R = \frac{1}{2^{N-1}
(N-1)!}\text{diag}\Bigl(\overbrace{\strut 1,\ldots,1}^{\mbox{$N$
times}},\frac{1}{N},\frac{1}{N+1},\ldots\Bigr).
\]
In practice, we observe that many other diagonal preconditioners also give a
bounded condition number and it is likely that there are preconditioners 
which give better practical bounds on the backward error, but only by a constant factor. 
 
The analysis of $(\ref{ODE})$ will follow from the fact that, on suitably
defined spaces,
$$\begin{pmatrix}
\mathcal{B}\\
\mathcal{L}\\
\end{pmatrix}
\R  =  I + \K,$$
for a compact operator $\K$, where $\mathcal{L}$ is the $N$th order differential
operator and $\mathcal{B}$ is a boundary operator representing $N$ boundary
conditions. To this aim, we need
to be precise on which spaces these operators act on. Since we are working in
coefficient space, we will consider the problems as
defined in $\ell_\lambda^2$ spaces:

\begin{definition}
	The space $\ell_\lambda^2 \subset {\mathbb C}^\infty$ is defined as the Banach space with
norm
	$$\|{\mathbf u}\|_{\ell^2_\lambda} = \sqrt{\sum_{k=0}^\infty |u_k|^2
(k+1)^{2\lambda}} < \infty.$$
\label{def:higherordernorms}
\end{definition}
We now show that the preconditioned operator is a compact perturbation of the identity.

\begin{lemma}
Suppose that the boundary operator $\mathcal B : \ell^2_D \rightarrow {\mathbb C}^K$ is bounded.  Then
\[
\begin{pmatrix}
\mathcal{B}\\
\mathcal{L}\\
\end{pmatrix}
\R  =  I + \K
\]
where $\K : \ell^2_\lambda \rightarrow \ell^2_\lambda$ is compact for $\lambda =
D-1,D,\ldots$.
\end{lemma}	
 \begin{proof} Firstly, note that
	\begin{align*}
		\begin{pmatrix}\B\cr\L\end{pmatrix} & =  \begin{pmatrix}2^{N-1}(N-1)!\P_N \cr \D_N\end{pmatrix} + \begin{pmatrix}\B -
2^{N-1}(N-1)!\P_N
\cr 0\end{pmatrix}\cr
		& + \begin{pmatrix}0\cr\S_{N-1}\M_{N-1}{[a^{N-1}]}
\D_{N-1} + \cdots + \S_{N-1} \cdots \S_1 \M_1{[a^1]} \D + \S_{N-1} \cdots \S_0
\M{[a^0]}\end{pmatrix}
	\end{align*}
where 
	$\P_N : \ell^2_\lambda \rightarrow {\mathbb C}^N$ is the $N \times
\infty$ projection operator (\ref{eq:projection}).

	Secondly, we remark that $\R : \ell^2_\lambda \rightarrow
\ell^2_{\lambda+1}$ and hence, $\B \R : \ell^2_\lambda \rightarrow {\mathbb C}^N$ is bounded for
$\lambda = D - 1, D,\ldots$.  Furthermore, $\S_k :\ell^2_\lambda
\rightarrow \ell^2_{\lambda+1}$ is bounded for $k = 1,2,\ldots$ and so is,  
$\D_N : \ell^2_\lambda \rightarrow \ell^2_{\lambda - 1}$. It follows that 
	$$\begin{pmatrix}2^{N-1}(N-1)!\P_N  \cr \D_N\end{pmatrix} \R = I :
\ell^2_\lambda
\rightarrow \ell^2_\lambda.$$

 Since $\begin{pmatrix}\B - 2^{N-1}(N-1)!\P_N \cr 0\end{pmatrix} \R :
\ell^2_\lambda
\rightarrow \ell^2_\lambda$ for $\lambda = D-1, D,\ldots$ is bounded and has
finite rank,  it is compact.  Note that ${\cal R}$ is compact as an operator $\R
: \ell^2_\lambda \rightarrow \ell^2_\lambda$,  as are $\S_{N-1},\ldots,\S_1$
(since $\S_k \R^{-1} : \ell^2_\lambda \rightarrow \ell^2_\lambda$ are bounded and
$\R$ is compact).  It follows that the last term
	$$\begin{pmatrix}0\cr\S_{N-1} \cdots  \S_0 \M{[a^0]}\end{pmatrix}{\cal
R} : \ell^2_\lambda \rightarrow \ell^2_\lambda$$
is compact, since $\M[a^0]$ and $\S_0$ are also bounded.  Finally, each of the
intermediate terms are compact since $\begin{pmatrix}0\cr\D_k\end{pmatrix} {\cal
R}: \ell^2_\lambda \rightarrow \ell^2_\lambda$ is bounded and $\S_k$ are
compact. 

We have shown that the preconditioned operator is the identity plus a sum of compact operators and hence a compact perturbation of the identity.
\end{proof}
 
 The compactness of $\K$ allows us to show well-conditioning and convergence.

% \begin{lemma}
% 	Suppose that $\begin{pmatrix} \B \cr \L\end{pmatrix} : \ell^2_\lambda
%\rightarrow \ell^2_{\lambda - 1}$ is an invertible operator for some $\lambda
%\in \{D-1,D,\ldots\}$.  Then
% 	$$\|{ I + P_n\K P_n^\top}\|_{\ell^2_\lambda}  = \O(1)\hbox{ and }\|{(I +
%P_n\K P_n^\top)^{-1}}\|_{\ell^2_\lambda}  = \O(1).$$
%%
% \end{lemma}
% \begin{proof}
% 	Since $\R : \ell^2_{\lambda -1 } \rightarrow \ell^2_\lambda$ is
%trivially invertible, we have that $I + \K : \ell^2_\lambda \rightarrow
%\ell^2_\lambda$ is invertible.  The lemma follows since $\K$ is compact. 

 \begin{lemma}
 	Suppose that $\begin{pmatrix} \B \cr \L\end{pmatrix} : \ell^2_{\lambda +
1} \rightarrow \ell^2_{\lambda}$ is an invertible operator for some $\lambda \in
\{D-1,D,\ldots\}$.  Then, as $n \rightarrow \infty$, 
 	$$\|A_n R_n \|_{\ell^2_\lambda}  = \O(1)\hbox{ and }\|{(A_n
R_n)^{-1}}\|_{\ell^2_\lambda}  = \O(1),$$
for the diagonal matrix $R_n = \P_n \R \P_n^\top$ and truncated spectral matrix
	$$A_n = \P_n\begin{pmatrix}\B \\ \L\end{pmatrix}\P_n^\top.$$
 \end{lemma}
 \begin{proof}
 	Since $\R : \ell^2_{\lambda } \rightarrow \ell^2_{\lambda + 1}$ is invertible, we have that $I + \K : \ell^2_\lambda \rightarrow
\ell^2_\lambda$ is invertible.  The lemma follows since $\K$ is compact,
$$A_nR_n = \P_n (I + \K) \R^{-1} \P_n^\top
\P_n \R \P_n^\top  = I _n+ \P_n \K \P_n^\top,$$
 and $\P_n^\top\P_n \K \P_n^\top\P_n$  converges in norm
to $\K$. 
\end{proof}

 \subsection{Convergence}%\label{secconv}
 
Denote the coefficients of the exact solution by $\mathbf{u}$, and note that vector $\mathcal{P}_n^\top\mathcal{P}_n\mathbf{u}$ agrees with $\mathbf{u}$ for the first $n$ coefficients and thereafter has zero entries.  We show that our numerical scheme converges at the same rate as $\mathcal{P}_n^\top\mathcal{P}_n\mathbf{u}$ converges to $\mathbf{u}$. 

%We first discuss convergence of $K_n$ to $\K$.
%
%\begin{lemma}
%	Suppose $a^N,\ldots,a^0$ are $m$ banded.  Then
%	$$\|(\K - P_n^\top K_n P_n) \mathbf v\|_{\ell^2_\lambda} \leq C \|
%\mathbf v - P_{n-m}^\top P_{n-m} \mathbf v\|_{\ell^2_\lambda}$$
%\end{lemma}
%\begin{proof}
%
%We note that 
%%
%	$$\M[a^0] - P_n^\top M[a] P_n$$
%%
%only operates on the coefficients after $n-m$, 
%
%\end{proof}

 \begin{theorem}
Suppose $\mathbf f \in \ell^2_{\lambda - N + 1}$ for some $\lambda \in
\{D-1,D,\ldots\}$, and that $\begin{pmatrix} \B \cr \L\end{pmatrix} :
\ell^2_{\lambda + 1} \rightarrow \ell^2_{\lambda}$ is an invertible operator. 
Define
 	$${\mathbf u}_n = A_n^{-1} \P_n\begin{pmatrix} \mathbf c \cr 
\S_{N-1}\cdots \S_0\mathbf{f} \end{pmatrix}.$$
 Then
 	$$\|{\mathbf u} - \P_n^\top {\mathbf u}_n\|_{\ell^2_{\lambda+1}} \leq C
\| {\mathbf u} - \P_n^\top \P_n {\mathbf u}\|_{\ell^2_{\lambda + 1}} \rightarrow
0.$$
\end{theorem}
\begin{proof}

Let $\mathbf v_n = R_n^{-1} \mathbf u_n$ and $\mathbf v = \R^{-1} \mathbf u$.
First note that 
\[
\begin{pmatrix} \B \\ \L \end{pmatrix}\mathbf{u} = \begin{pmatrix} \B \\ \L \end{pmatrix}\mathcal{R}\mathcal{R}^{-1}\mathbf{u} = \left(I+\K\right)\mathbf{v}  = \begin{pmatrix}\mathbf c \\  \S_{N-1}\cdots \S_0\mathbf{f}
\end{pmatrix} \in \ell_\lambda^2
\]

and that 
\[
\mathbf v_n = (A_nR_n)^{-1} \P_n (I + \K)\mathbf v 
\]
%
%	$$ \mathbf v_n = (A_nR_n)^{-1} \P_n (I + \K)\mathbf v.$$
Moreover, since $A_n R_n =   \P_n (I + \K) \P_n^\top$,  
	$$\P_n {\mathbf v} = (A_n R_n)^{-1}  \P_n (I + \K) \P_n^\top \P_n
{\mathbf
v},$$
and thus we have 
\begin{align*}
{\mathbf v} - \P_n^\top {\mathbf v}_n &
= {\mathbf v} - \P_n^\top \P_n \mathbf v + \P_n^\top(A_n R_n)^{-1} \P_n ( I +
\K)
\P_n^\top \P_n \mathbf v \cr 
& \qquad \qquad \qquad \qquad \qquad \qquad \qquad- \P_n^\top(A_n R_n)^{-1} \P_n
( I + \K){\mathbf v} \cr
& = (I - \P_n^\top (A_n R_n)^{-1}  \P_n (I + \K)) ( {\mathbf v} -
\P_n^\top \P_n {\mathbf v}).
\end{align*}

Finally, we  use the fact that $\|{\R^{-1} \mathbf u}\|_{\ell^2_\lambda} \leq
\frac{1}{N}\|\mathbf u
\|_{\ell^2_{\lambda + 1}}$ and $\| \R \mathbf v \|_{\ell_{\lambda + 1}} \leq (N
+ 1) \| \mathbf v\|_{\ell^2_\lambda}$ to bound the error in the solution by the error in the Chebyshev series of $u$ and its truncation,
	\begin{align*}
		\|\mathbf u - \P_n^\top \mathbf u_n\|_{\ell^2_{\lambda + 1}}
&\leq (N +
1)\|\mathbf v - \P_n^\top \mathbf v_n\|_{\ell^2_\lambda}  \\
		&\leq (N + 1)\left[1 + \left\|(A_n
R_n)^{-1}\right\|_{\ell^2_\lambda} \left(1 + \|
\K\|_{\ell^2_\lambda}\right) \right] \left\| \mathbf v - \P_n^\top \P_n \mathbf
v\right\|_{\ell^2_\lambda}
\\
		& \leq C  \left\| \mathbf u - \P_n^\top \P_n \mathbf
u\right\|_{\ell^2_{\lambda + 1}}.
	\end{align*} 
Since $\mathbf u \in \ell_{\lambda+1}^2$ we know that $\left\| \mathbf u -
\P_n^\top \P_n \mathbf
u\right\|_{\ell^2_{\lambda + 1}}\rightarrow 0$ as $n\rightarrow \infty$. 
\end{proof}

	%\| (P_n(I + \K) P_n^\top)^{-1} \| \|I + \K\|  \|{\mathbf u} - R_n P_n
%{\mathbf u}\| $$
%

\section{Fast linear algebra for almost banded matrices} \label{sec:fastsolver}

\def\msup{m_{\rm R}}
\def\msub{m_{\rm L}}

The spectral method we have described requires the solution of a linear system
$Ax=b$
where $A\in\mathbb{C}^{n\times n}$. The matrix $A$ is banded, with $\msup =\O(m)$ non-zero superdiagonals and $\msub =\O(m)$ non-zero subdiagonals (so that $m = \msub + \msup + 1$), except for the
first $K$ dense boundary rows.
% Here, $\msup \geq K$ and
%$m_2\geq-m_1$ and, for simplicity, we assume that $m_2\geq 0$.
 The typical
structure of $A$ is depicted in Figure
\ref{fig:Filled}.  Here, we describe a stable algorithm to solve $Ax=b$
in $\O(m^2n)$ operations and with space requirement $\O(mn)$.  

\begin{figure}
\centering
\includegraphics[width=\textwidth]{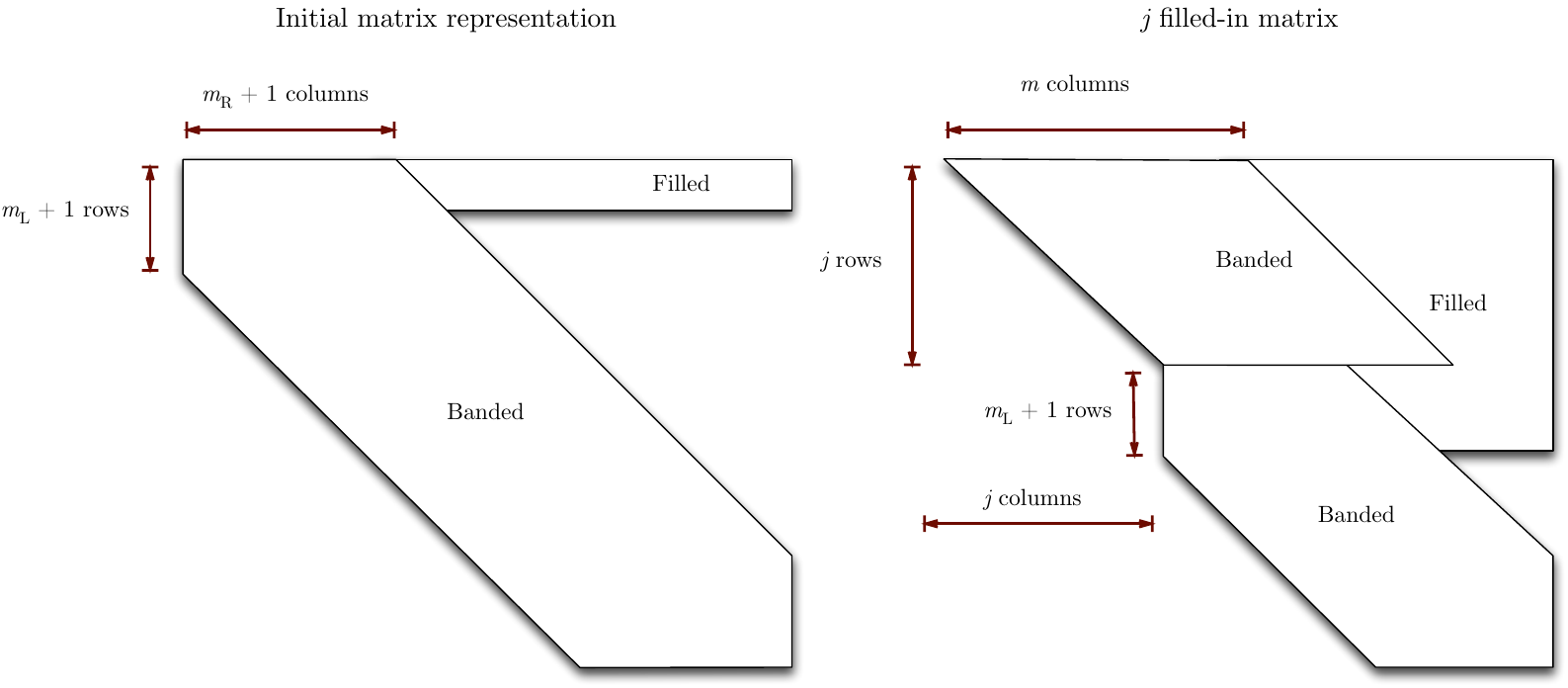}

\caption{\label{fig:Filled} Structure of operators during QR factorization. Left: Depicts the structure of the original differential operator. Right: Depicts a $j$ filled-in matrix obtained after upper triangularizing the first $j$ columns. }
\end{figure}

We will  solve $Ax=b$ by computing a $QR$
factorization using Given's rotations. However, the resulting upper triangular part will be
dense because of \emph{fill-in} caused by the boundary rows.  We will show that these dense rows can still be represented sparsely, and that the resulting upper triangular matrix can be solved in $\mathcal{O}(m^2n)$ operations.  

\begin{remark} \hbox{ }
Alternatively, the $A=QRP^*$ decomposition can be constructed in $\mathcal{O}(m^2n)$ operations by apply Given's rotations to the left and the right, which prevents the boundary row(s) causing fill-in \cite{Chandrasekaran_03_01}. However, with this decomposition it is unclear whether the optimal truncation $n_{\rm opt}$ can be determined in $\O(n_{\rm opt})$ operations.

\end{remark}

\subsection{QR factorization for filled-in matrices}\label{sec:qrfilled}

	Represent the matrix $A$ after the $j$th stage of the QR decomposition, where the $j$th
column has been completely reduced, by $B$.  We claim that $B$ has the form of a $j$ {\it filled-in
matrix}. 
	
\def\ee{{\bf e}}	
	
\begin{definition}
	 $B$ is a $j$ filled-in matrix if, for $k = 1,\ldots, j$,  the $k$th row of $B$ has the
form 
%
%	
%%	 zero entries up to the $k$th column, followed by $m  = m_2 - m_1$ entries, followed by
%entries which  a linear combination of the initial $K$ boundary rows to represent the fill-in of the
%matrix, where the linear combination {\it is the same} for each column.  In other words, the $k$th
%row of the matrix has the following form
%
	\begin{align}\label{eq:upperrow}
	{\bf e}_k^\top B = \Biggl(& \overbrace{0,\ldots,0}^{k - 1 \hbox{ times}},
\overbrace{B_{k,k},\dots, B_{k,k+m - 1}}^{\hbox{banded terms}}, \overbrace{{\bf b}_k^\top {\cal
B}_{1:K,k+m:n} }^{\hbox{fill-in terms}}\Biggr),
	\end{align}
where ${\bf b}_k \in {\mathbb C}^K$.  
%
%
%
%%Leaving the boundary rows unevaluated, this takes $m +K$ complex numbers to represent:
%%$B_{k,k},\ldots,B_{k,k+m}$ and ${\bf b}_k$.  
%
%
Furthermore, every row $k = j +1,\ldots,j+\msub + 1$ has the form
	\begin{align*}%\label{eq:lowerrow}
	{\bf e}_k^\top B = 	\Biggl( \overbrace{0,\ldots,0}^{j \hbox{ times}},
\overbrace{B_{k,j+1},\dots, B_{k,k + \msup}}^{\hbox{banded terms}}, \overbrace{{\bf b}_k^\top {\cal
B}_{1:K,k+\msup+1:n}  }^{\hbox{fill-in terms}}\Biggr). 
	\end{align*}
%
 %This takes $k - j + m_2  + K < m + K$ complex numbers to represent.  
 The remaining rows have the form
  	\begin{align}\label{eq:unprocrow}
	{\bf e}_k^\top B = 	\Biggl( \overbrace{0,\ldots,0}^{k - \msub - 1 \hbox{ times}},
\overbrace{B_{k,k - \msub},\dots, B_{k,k + \msup}}^{\hbox{banded terms}}, \overbrace{0,\ldots,0}^{\hbox{$n - k - \msup$ times}}\Biggr). 
	\end{align}
% 	\begin{align}\label{eq:unprocrow}
%	{\bf e}_k^\top B = 	\Biggl( \overbrace{0,\ldots,0}^{k - \msub - 1 \hbox{ times}},
%\overbrace{B_{k,k - \msub},\dots, B_{k,k + \msup}}^{\hbox{banded terms}}, \overbrace{{\bf b}_k^\top {\cal
%B}_{1:K,k+\msup+1:n}  }^{\hbox{fill-in terms}}\Biggr). 
%	\end{align}
 %
% and   take $m < m + K$ complex numbers to represent.  

\end{definition}
	See Figure~\ref{fig:Filled} for a depiction of a $j$ filled-in matrix.   Essentially, it is
a banded matrix where the top-right part can be filled with linear combinations of the
boundary rows.   Note that each row of a filled-in matrix takes at most $m +K$ entries to represent,
a bound which is independent of $j$.  Since the QR factorization only applies Given's rotations on the left (i.e. linear combinations of rows), and the initial matrix is a $0$ filled-in matrix, the elimination results in $j$ filled-in matrices.  This means that throughout the elimination, the matrix can be represented with $\mathcal{O}(mn)$ storage.  

	In section \ref{sec:opttrunc}, we perform QR factorization adaptively on the operator.  This is successful because the representation as a $j$ filled-in matrix is independent of the number of columns, and the rows below the $(j + \msub)$th are unchanged from the original operator, hence can be added during the factorization.

\subsection{Back substitution for filled-in matrices}

	After we have performed $n$ stages of Given's rotations, the first $n$ rows are of the form
\eqref{eq:upperrow} and hence are {\it upper triangular}.  Thus, we can now perform back substitution, by truncating the right-hand side. 
The last $m$ rows consist only of banded terms, and standard back substitution is used to calculate
$u_{n-m + 1},\ldots,u_n$.  We then note that the $k$th row imposes the condition
$$
		B_{k,k} u_k  =  c_k - \sum_{s=1}^mB_{k,k+s} u_{k+s}  
			 - {\bf b}_k^\top \sum_{s=k+m+1}^n {\cal B}_{1:K,s} u_{s}
$$
on the solution.  We can thus obtain an $\O(mn)$ back substitution algorithm by the following: 
%
%\begin{align*} 
\begin{algorithm}[H]
\caption{Back substitution for filled-in matrices}
\begin{algorithmic}
\State $\mathbf{p}_{n-m}=\mathbf 0$ \qquad\qquad\qquad\qquad\qquad\qquad\hbox{(i.e., a vector of $K$ zeros)}%${\br r}_{n-m}=0$
\For {$k=n-m-1 \to 1$}
\State $\mathbf{p}_k = {u_{k+m+1} \mathcal{B}}_{1:K,k+m+1}  + \mathbf{p}_{k + 1}$
\State $u_k = \frac{1}{B_{k,k}}\left(c_k - \sum_{s=1}^mB_{k,k+s} u_{k+s} - {\mathbf{b}}_k^\top {\mathbf{ p}}_k\right)$
\EndFor
	%{\cr
	% %\right).
%\end{align*}
\end{algorithmic}
\end{algorithm}
This is mathematically equivalent to standard back substitution.  However, the reduced number of
operations decreases the accumulation of round-off error.

%When we perform row-reduction of the first column for the first $K$ rows, we are simply altering
%the linear combination of the boundary rows, hence only changing the constants $b_{i,j}$ for $i,j =
%1,\ldots, K$.  
%
%%	We then reach the $(K + 1)$th row, corresponding to the first row of the differential
%operator.  Only the first $m_2$ entries of this row are non-zero.  We apply the computed Given's
%rotation to these entries, resulting in the first row having $m_2$ unaltered entries.  We then apply
%the Given's rotation between the filled-in rows.  This only alters the constants
%$b_{1,1},\ldots,b_{1,K}$ and $b_{K+1,1},\ldots,b_{K+1,K}$.  
%%	
%%	We continue the process until all $-m_1$ nonzero rows of the first column are eliminated. 
%The final representation has $m_2 - m_1$  altered entries in the first row, with the right bandwidth
%of the other rows unchanged (though their filled-in constants altered).  We now row reduce the
%remaining columns.  It is clear that the same logic as the first column applies in each case,
%showing that the matrix has the claimed form.  
%%	
%	
%\sotodo{Since this is a crucial part of the algorithm, the paragraphs above need rewording}

\subsection{Optimal truncation}\label{sec:opttrunc}

	One does not know {\it apriori} how many coefficients $n_{\rm opt}$ are required to resolve the solution to relative machine precision.  A
straightforward algorithm for finding $n_{\rm opt}$ is to continually double the discretization size until the difference in the
computed coefficients is below a given threshold, which will result in an $\O(n_{\rm opt} \log
{n_{\rm opt}})$ algorithm.  We will present an alternative approach that achieves the optimal $\O(n_{\rm opt})$ complexity.  
	
	For the simple equation
	$$u' + u = f \qqand u(1) = c,$$
the truncation of the operator is tridiagonal with a single dense boundary row.  This
is equivalent to an {\it inhomogeneous three-term recurrence boundary value problem}.  The problem of
adaptively truncating such recurrence relationships is solvable by (F. W. J.) Olver's algorithm
\cite{OlversAlgorithm}.  The central idea is to apply row reduction (without pivoting) adaptively. 
The row-reduction applied to the right-hand side, when combined with a concurrent adaptive
computation of the homogeneous solutions to the recurrence relationship, allowed an explicit bound
for the {\it relative error}.  An alternative (and simpler) bound for the {\it absolute error} was
 obtained in examples.  The case of dense boundary rows, with an application to the
Clenshaw method \cite{ClenshawTauMethod} as motivation, was also considered.  Adaptation
of the algorithm to higher order difference equations was developed by Lozier \cite{Lozier_80}.

	We  adapt this approach to our case by incorporating it into the QR factorization of
subsection \ref{sec:qrfilled}.  By using QR factorization in place of Gaussian elimination, we avoid
potential numerical stability issues of the original Olver's algorithm in the low order
coefficients.   The key observation is that the boundary terms ${\cal B}$ and the rows of the form
\eqref{eq:unprocrow} can be evaluated lazily, with bounded computational cost per entry.  Thus the
truncation parameter $n$ is not involved in the proposed QR factorization algorithm (only in the
back substitution), hence the optimal truncation $n_{\rm opt}$ can be found adaptively.

	Represent the  Given's rotations that reduces the first $n$ columns of ${\cal A}$ by the
orthonormal operator $Q_n \in {\mathbb C}^{(n + \msub) \times (n + \msub)}$,  so that 
	$$\begin{pmatrix}{ Q}_n^\star \cr & I\end{pmatrix} {\cal A}  =  \begin{pmatrix}R_n & {\cal F} \cr & {\cal W}\end{pmatrix}$$
where $R_n \in {\mathbb C}^{n \times n}$ is upper triangular.   Our right-hand side is 
	$$\begin{pmatrix}{ Q}_n^\star \cr & I\end{pmatrix} \begin{pmatrix}{\bf c} \\ \S_{N-1}\cdots \S_0\mathbf{f}\end{pmatrix} =
\begin{pmatrix} \mathbf r_{1:n}\\\mathbf r_{n+1:n+\msub}\\ \mathbf 0 \end{pmatrix}$$
where we use the fact that $\mathbf f$ has only finite number of nonzero entries and assume that $n + \msub$ is greater than the number of nonzero entries.  Thus our
numerical approximation ${\mathbf u}_n = R_n^{-1} \mathbf r_{1:n}$ has the {\it forward error}:
	$$\begin{pmatrix}{ Q}_n^\star \cr & I\end{pmatrix}  {\cal A} \begin{pmatrix} {\mathbf u}_n \\ {\bf 0}\\{\bf 0}\end{pmatrix}  - \begin{pmatrix}
\mathbf r_{1:n}\\\mathbf r_{n+1:n+\msub}\\ {\bf 0} \end{pmatrix} =  \begin{pmatrix} R_n {\mathbf u}_n \\ {\bf 0}\\{\bf 0}
\end{pmatrix} - \begin{pmatrix} \mathbf r_{1:n}\\\mathbf r_{n+1:n+\msub}\\ {\bf 0} \end{pmatrix} =
\begin{pmatrix} {\bf 0} \\\mathbf r_{n+1:n+\msub}\\ {\bf 0} \end{pmatrix}.$$
Since we calculate $\mathbf r_{n+1:n+\msub}$ during the algorithm, we know the forward error {\it
exactly}.  Having a bound on $\|{\cal A}^{-1}\|$ allows us to bound the {\it backward error}; i.e.,
$\|{\mathbf u} - {\cal P}_n^\top {\mathbf u}_n\|$.   This can be improved further  by adapting the procedure of \cite{OlversAlgorithm,Lozier_80}, which in a sense
calculates an alternative bound based on the homogeneous solutions of the difference equation  as part of the algorithm.   % \sotodo{error is I?}

%For
%simplicity however, we assume that  $\|{\cal A}^{-1}\| < 100$.  \sotodo{why 100}?

\begin{figure}
  \centering
    \begin{minipage}[b]{.49\textwidth}
    \includegraphics[width=\textwidth]{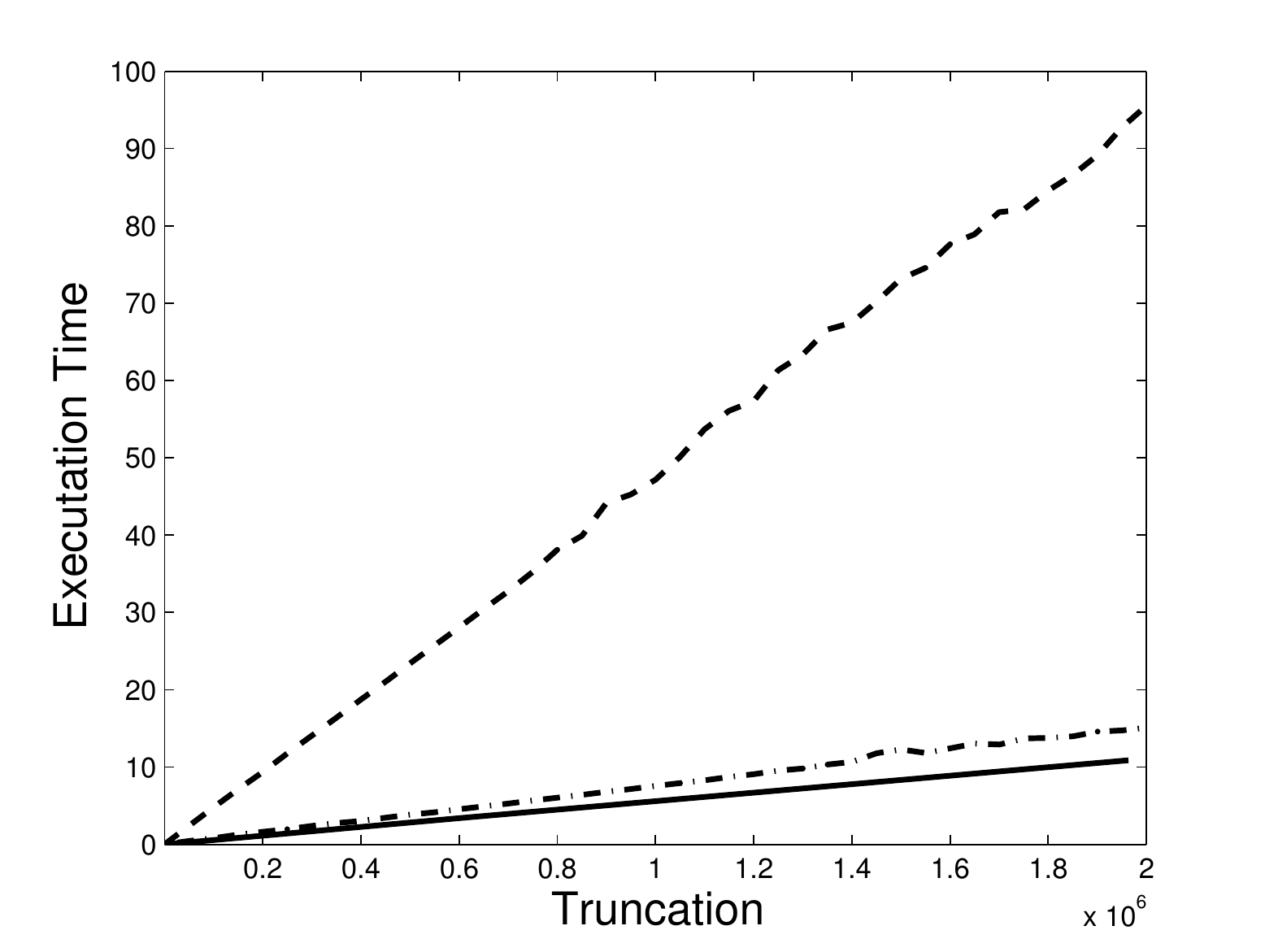}
\end{minipage}
\begin{minipage}[b]{.49\textwidth}
\includegraphics[width=\textwidth]{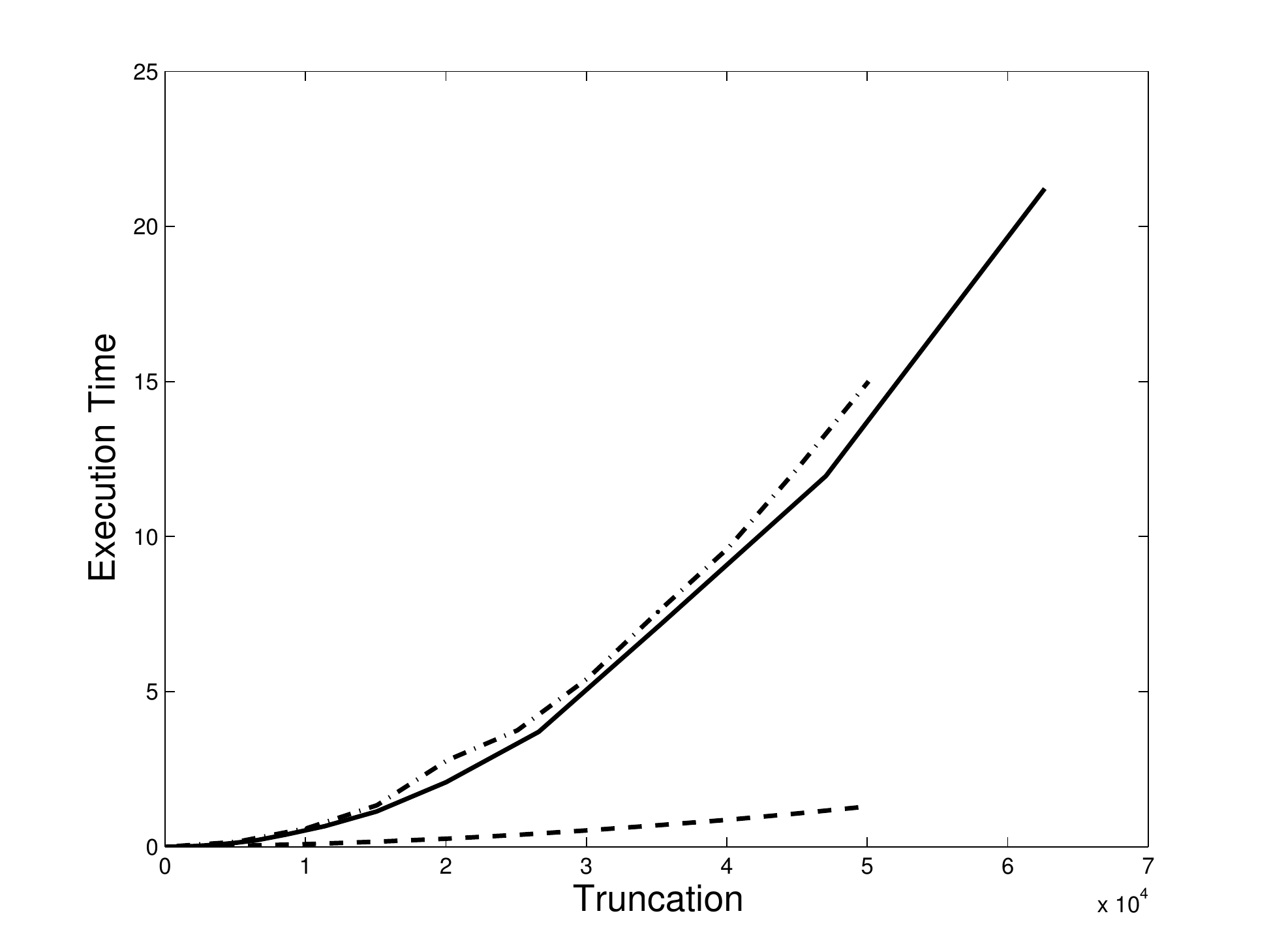}
\end{minipage}
\caption{Time versus optimal truncation, showing linear computational cost for the {\sc C++} implementation (left) and the {\sc Matlab} implementation (right). Examples are \eqref{eq:airyeqn} (solid),  \eqref{eq:smallband}  (dot-dashed) and \eqref{eq:cosband} (dashed).}
\label{fig:adaptiveqrtiming}
\end{figure}	

 In Figure~\ref{fig:adaptiveqrtiming}, we apply the adaptive QR decomposition to solve the Airy equation of example \eqref{eq:airyeqn} with $\epsilon = 1,10^{-1},10^{-2},\ldots,10^{-13}$, as well as
	\begin{align}
		u'' + (7 + 2x + 6 x^2) u &= \sum_{k=0}^\nu T_k(x) \hbox{ for } u(-1) = 1, u(1) = 1,  \label{eq:smallband}\\
		u'' + \cos x u &= \sum_{k=0}^\nu T_k(x) \hbox{ for } u(-1) = 1, u(1) = 1  \label{eq:cosband}
	\end{align}
for increasing values of $\nu$, up to 2 million.  In the last example, we replace $\cos x$ with its 13 point Chebyshev interpolating polynomial.    We plot the
number of seconds the calculation takes versus the adaptively calculated optimal truncation $n_{\rm opt}$.  
This demonstrates the $\O(n_{\rm opt})$ complexity of the algorithm, and the fact that the algorithm easily
scales to more than a million unknowns.   It also shows that, while $\O(n_{\rm opt})$ complexity is maintained, the computational cost does increase with the bandwidth of the variable coefficient.  In the Airy example \eqref{eq:airyeqn}, the time taken for $\epsilon = 10^{-13}$ is less than 11 seconds\footnote{CPU times were calculated on a 2011 iMac, with a 2.7 Ghz Intel Core i5 CPU}, with the  $n_{\rm opt}$ calculated to be approximately 2 million.  For example \eqref{eq:cosband}, a calculation resulting in  $n_{\rm opt}$ being 2 million increases the timing to 95 seconds.

	On the right of Figure~\ref{fig:adaptiveqrtiming} we plot the timing of {\sc Matlab}'s built-in sparse LU solver applied to the truncated equation, with  $n = n_{\rm opt}$ pre-specified.  Even without the added difficulty of calculating $n_{\rm opt}$,  the computational cost grows faster than $\O(n)$, prohibiting its usefulness for extremely large $n$.  

\begin{remark}
	We calculated Figure~\ref{fig:adaptiveqrtiming} using {\tt C++}.  There is a great deal of room for optimizing the implementation, as we do not use GPU, parallel or vector processing units.  
\end{remark}

\subsection{Linear algebra stability in higher order norms}%\label{sec:higherordernorms}

We first remark that, if $\B$ is a bounded operator from $\ell^2_1 \rightarrow
\ell^2_0$, such as Dirichlet boundary conditions, then the results of section
\ref{sec:preconditioner} prove that the \emph{preconditioned} linear system has
bounded $2$-norm condition number as $n \rightarrow \infty$.  Because the  $Q R
$
decomposition is computed using Givens rotations, which are stable in
$\ell^2_0$ \cite{Higham_02_01}, as is backward substitution, we see that the
linear algebra scheme applied to the preconditioned operator is stable, and has
$\O(m^2 n)$ complexity.  

%We remark that a variant of Lemma
%\ref{unprevspre} can be proved showing that this stability property holds for
%the non-preconditioned version as well. 

	The results of section \ref{sec:preconditioner} also show convergence
and well-conditioning in high order norms.  One would expect numerical
round-off in $Q R$ decomposition to destroy this convergence property. 
However, in practice, this is not the case.  In Figure
\ref{fig:higherordernorms}, we solve the standard Airy equation as a two point
boundary value problem:
%\begin{equation}\label{eq:standardairyequation}
\[
u'' - x u = 0, \quad u(-1) = \Ai(-1)\hbox{ and } u(1) = \Ai(1).
\]
We witness convergence in higher order norms as well.  In other words,
the computed solution has a fast decaying tail, and all of the \emph{absolute}
error is in the low order coefficients.  

\begin{figure}
  \begin{minipage}[b]{.5\textwidth}
    \includegraphics[width=\textwidth]{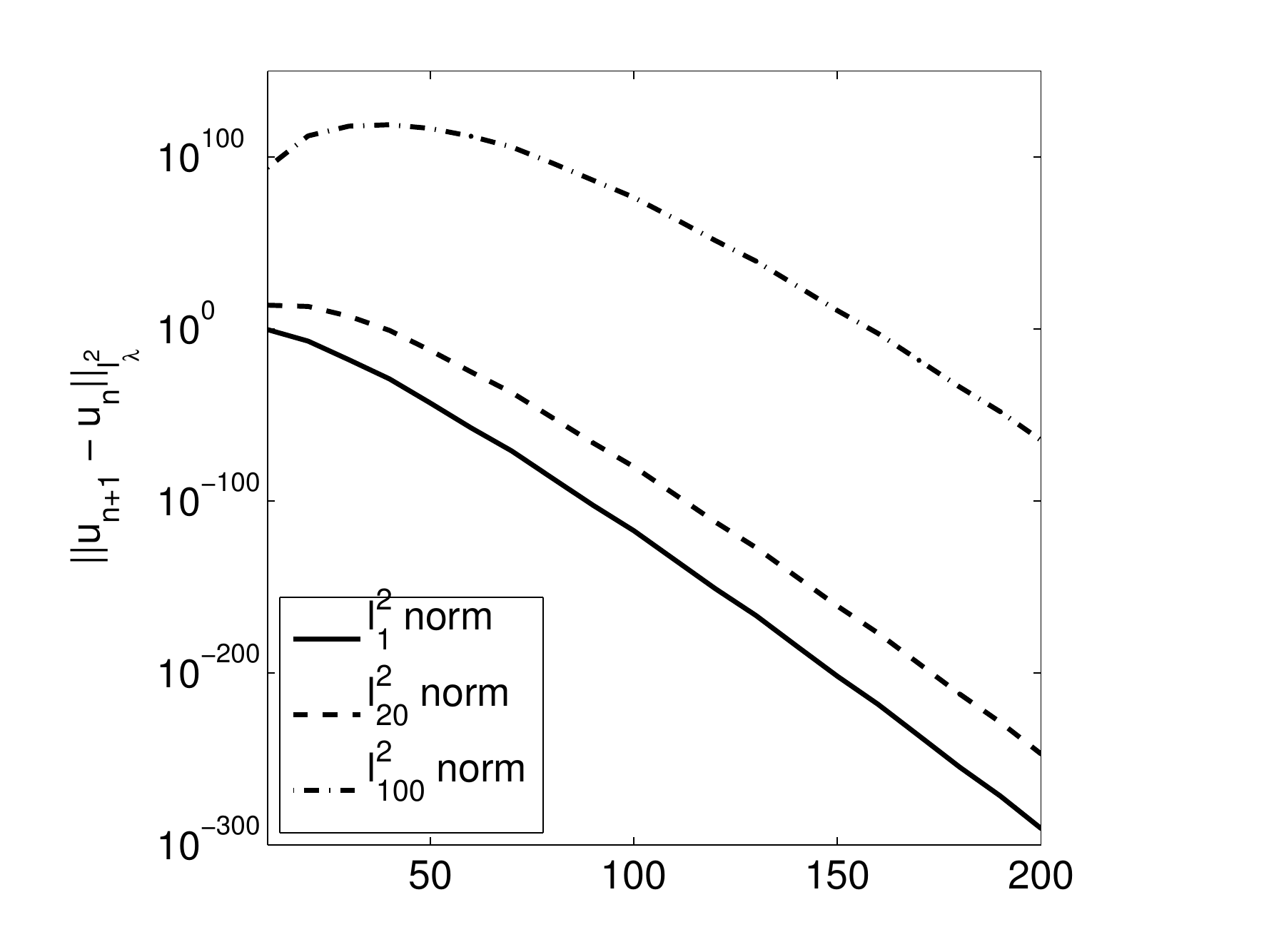}
\end{minipage}
\begin{minipage}[b]{.5\textwidth}
\includegraphics[width=\textwidth]{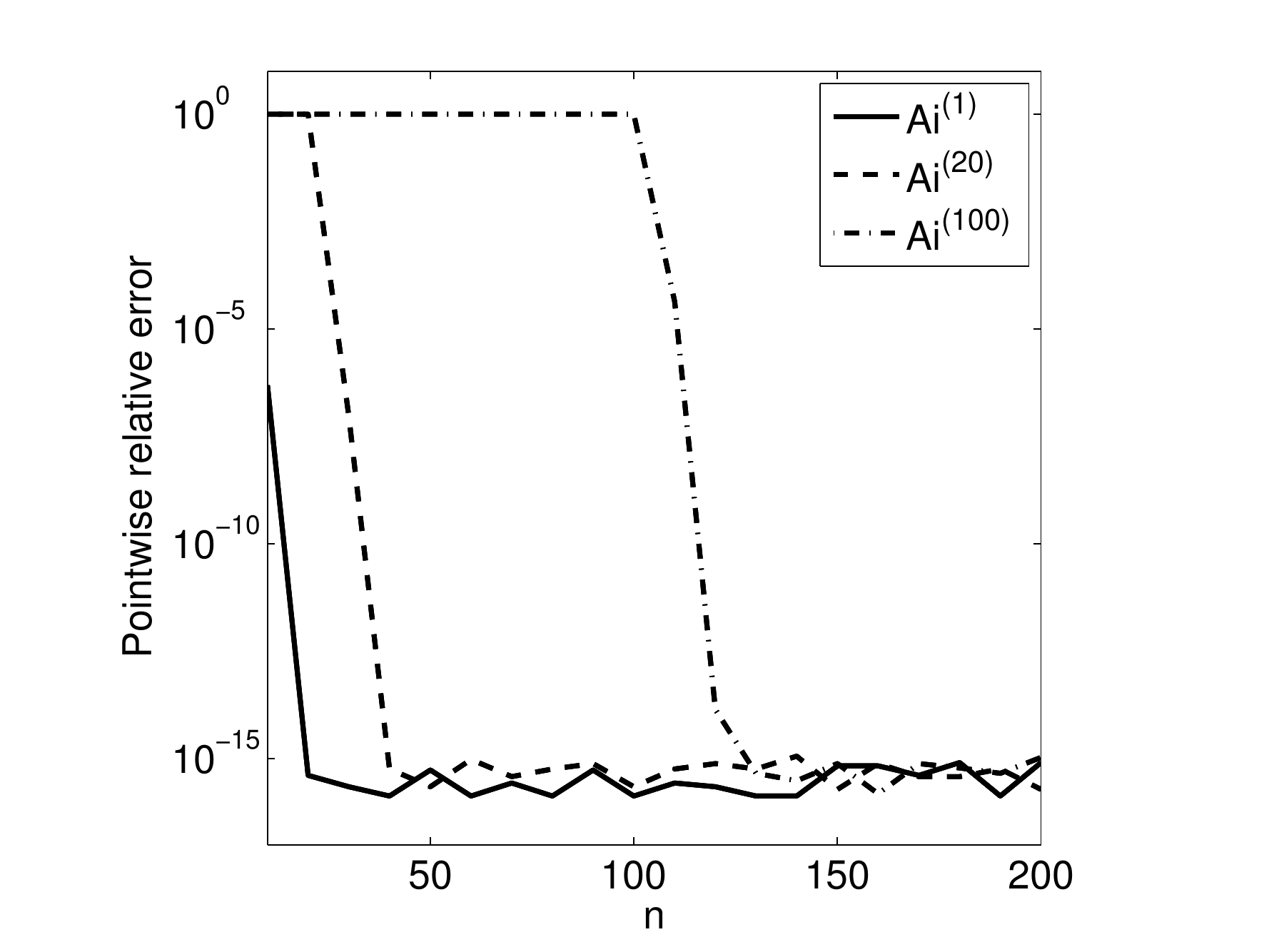}
\end{minipage}

\caption{Left: The Cauchy error of the solution coefficients measured
in the: (solid) $\ell_{1}^2$-norm, (dashed) $\ell_{20}^2$-norm, (dot-dashed)
$\ell_{100}^2$-norm. Right: The relative error in derivatives of the
solution at $x=-\frac{1}{2}$ for the $1$st (solid), $20$th (dashed), $100$th (dot-dashed) derivative. This plot shows that for sufficiently large discretisations very high derivatives of the solution can be resolved.} 
\label{fig:higherordernorms}
\end{figure}

	Convergence in higher order norms implies convergence of derivatives,
and this is verified by differentiating the computed solution by applying ${\cal
S}_0^{-1} {\cal D}_1$ repeatedly. (We note that the banded, upper triangular nature of
${\cal S}_0$ means that its inverse applied to a vector is computable in $\O(n)$
operations; after all, this corresponds to differentiating without converting bases which can be done in $\mathcal{O}(n)$ operations \cite{Mason_03_01}.)  We compare the computed derivatives with the true
derivatives of the Airy function $\Ai^{(p)}$ at a single point, $x = -1/2$, and
witness convergence for $p = 1,5$ and $20$.

	We note that the stability of the $Q R$ algorithm  in higher
order norms appears to follow from $Q$ being almost banded: it is banded along the
superdiagonal, and  decays exponentially along the subdiagonals.    In the case where the boundary conditions are such that $A$ itself is banded, exponential decay in $Q = A R^{-1}$ follows since  $R^{-1}$ has exponentially decay, due to $R$ being banded and well-conditioned \cite{Demko_84_01}.  

% This can be adapted to other boundary conditions by noting that $R$ is a filled-in matrix; i.e.,  a semi-separable perturbation of a banded matrix.  
	
	Since $\ell^2_\lambda$ is a Hilbert space, Givens rotations can be
constructed with the relevant inner product, resulting in orthogonal operations
(i.e., with condition number one) in $\ell^2_\lambda$.  The stability of such an
algorithm in $\ell^2_\lambda$ follows immediately.  With this modification, we
have an $\O(m^2 n)$ stable algorithm which is guaranteed to converge in higher order norms.  
	
\begin{remark}
	While we have discussed the convergence in higher order norms with
Dirichlet boundary conditions, the exact same logic applies to the convergence
and stability observed with higher order boundary conditions, such as Neumann conditions.
\end{remark}

\section{Future work}

We have designed a spectral method that achieves $\O(n)$ computational cost, stability and generality for solving linear ODEs.  We determined the optimal truncation adaptively using the QR factorization.   We believe that the ideas introduced in this paper will serve as a basis for future spectral methods. 

An exciting generalization of this work will be to higher
dimensions,
where the density of matrices has inhibited the usefulness of spectral methods.   A similar approach, based on boundary recombination,    was used in
\cite{Shen_95_01} for
the Helmholtz equation.
Adapting our method to rectangular domains, using
tensor products of ultraspherical polynomials, results in a tensor product of almost banded matrices.     Constructing an adaptive QR decomposition to such matrices will be crucial for achieving  competitive computational costs, and for optimally choosing the number of unknowns needed in each dimension.  

   Using the theory of \cite{Ryland_11_01}, there are
potentially generalization of ultraspherical polynomials to deltoid domains.   
What is less clear is how the results would be generalizable to other
domains, such as the the triangle.

For problems with boundary layers, or localized oscillatory behaviour, it can be more efficient to 
subdivide the unit interval, in order to minimize the total number of unknowns.  This will be of fundamental importance for PDEs, where  solutions typically have singularities at corners.  In the 1D case, it is straightforward to incorporate subdivision by representing the operators as block matrices, with additional boundary rows to impose continuity.  Whether the adaptive QR decomposition can be easily generalized to such matrices is less clear.

	Finally, an extension of this work is to nonlinear
differential
equations, of the form
	$$\B u = \mathbf 0 \qqand \L u + g(u) = f.$$
  Our approach can be incorporated  into an
infinite-dimensional Newton iteration, {\it \'a la} \cite{Birkisson_11,Driscoll_08_01}.  The  Newton
iteration takes the form
 	$$u_{k+1} = u_k + \begin{pmatrix} \B \\ \L + g'(u_k)\end{pmatrix}^{-1}\begin{pmatrix} \mathbf 0 \\ \L u_k + g(u_k) - f \end{pmatrix}.$$
Since the linear operator that is inverted involves the solution itself,  the bandedness of multiplication is lost, at least when na\"ively implemented.  It may be possible to overcome this difficulty by using the fact that the derivative for Newton iteration need not be accurate to machine precision.  Moreover, the  decay in the coefficients of the operator can combine with decay in the solution, hence the entries of the operator can be truncated more aggressively while maintaining accuracy.  However, even with a dense representation of the operator, the
stability of the resulting algorithm is preserved in initial numerical experiments.

\section*{Acknowledgments}
	We thank P. Gonnet, discussions with whom led to the observation of
well-conditioning of coefficient methods, which initiated the research of
this
paper.  We also thank the rest of the Chebfun team, including T. A. Driscoll,
N. Hale and L. N. Trefethen for valuable feedback.  We thank J. P. Boyd and the anonymous reviewers for very useful comments and suggestions.  We finally thank D. Lozier and F. W. J. Olver for discussions and references related to Olver's algorithm.

\end{document}